\documentclass[12pt]{amsart}

\usepackage{amsfonts}
\usepackage{amsmath}
\usepackage{amssymb}
\usepackage{amsthm}
\usepackage{enumerate} 
\usepackage{hyperref}
\hypersetup{
	colorlinks=true,
	citecolor=blue
}
\usepackage[margin=1.2in]{geometry}
\usepackage{mathrsfs} 
\usepackage{graphicx} 
\usepackage{tikz-cd}

\usepackage{xcolor}

\theoremstyle{plain}
\newtheorem{theorem}{Theorem}[section]
\newtheorem{lemma}[theorem]{Lemma}
\newtheorem{proposition}[theorem]{Proposition}
\newtheorem{corollary}[theorem]{Corollary}

\theoremstyle{definition}
\newtheorem{definition}[theorem]{Definition}

\newtheorem{remark}[theorem]{Remark}
\newtheorem{problem}[theorem]{Problem}

\newtheorem{example}[theorem]{Example}

\numberwithin{equation}{section}

\newcommand\N{\mathbb{N}}

\newcommand\R{\mathbb{R}}

\newcommand\C{\mathbb{C}}

\newcommand\ev[2]{\langle#1,#2\rangle}

\newcommand\SES[5]{\begin{tikzcd}[ampersand replacement=\&] 0 \arrow{r} \& #1 \arrow{r}{#4} \& #2 \arrow{r}{#5} \& #3 \arrow{r} \& 0\end{tikzcd}}

\DeclareMathOperator\im{im}
\DeclareMathOperator\id{id}

\DeclareMathOperator\condS{S}

\DeclareMathOperator\condWS{WS}
\DeclareMathOperator\condQ{Q}

\DeclareMathOperator\condA{A}
\DeclareMathOperator\conduA{\underline{A}}
\DeclareMathOperator\DN{DN}
\DeclareMathOperator\uDN{\underline{DN}}
\DeclareMathOperator\DOmega{D\Omega}

\DeclareMathOperator\Proj{Proj}
\DeclareMathOperator\Ind{Ind}
\DeclareMathOperator\Ext{Ext}

\begin{document}

\title[An extension result for $(LB)$-spaces]{An extension result for $(LB)$-spaces and the surjectivity of tensorized mappings}

\author[A. Debrouwere]{Andreas Debrouwere}
\address{Department of Mathematics and Data Science \\ Vrije Universiteit Brussel, Belgium\\ Pleinlaan 2 \\ 1050 Brussels \\ Belgium}
\email{andreas.debrouwere@vub.be}

\author[L. Neyt]{Lenny Neyt}
\thanks{L. Neyt gratefully acknowledges support by the Alexander von Humboldt Foundation and by FWO-Vlaanderen through the postdoctoral grant 12ZG921N }
\address{Universit\"{a}t Trier\\ FB IV Mathematik\\ D-54286 Trier\\ Germany}
\email{lenny.neyt@UGent.be}


\begin{abstract}
We study an extension problem for continuous linear maps in the setting of $(LB)$-spaces. More precisely, we characterize the pairs $(E,Z)$, where $E$ is a locally complete space with a fundamental sequence of bounded sets and $Z$ is an $(LB)$-space, such that for every exact sequence of $(LB)$-spaces
$$
				 \SES{X}{Y}{Z}{\iota}{} 
				 $$
the map $$
 L(Y,E) \to L(X, E), ~ T \mapsto T \circ \iota
$$
is  surjective, meaning that each continuous linear map $X \to E$ can be extended to a continuous linear map $Y \to E$ via $\iota$, under some mild conditions on $E$ or $Z$ (e.g. one of them is nuclear).				
We use our extension result to obtain sufficient conditions for the surjectivity of tensorized maps between Fr\'{e}chet-Schwartz spaces. As an application of the latter, we study vector-valued Eidelheit type problems. Our work is inspired by and extends results of Vogt \cite{V-TensorFundDFRaumFortsetz}. 
\end{abstract}

\maketitle  

\section{Introduction}
Problems concerning the extension of continuous linear maps between locally convex spaces are a well-studied topic in functional analysis. In this article, we study an extension problem for continuous linear maps in the context of $(LB)$-spaces that was treated by Vogt in the unpublished manuscript \cite{V-TensorFundDFRaumFortsetz}  (see also the work of Petzsche \cite{Petzsche}). We extend his results in several directions. The main motivation for our work stems from the fact that it may be used to give sufficient conditions ensuring the surjectivity of tensorized maps between Fr\'{e}chet-Schwartz spaces, as we now proceed to explain.


Consider an exact sequence of locally convex spaces
\begin{equation}
\label{seq-1}
 \SES{F}{G}{H}{}{Q} 
 \end{equation}
and let $E$ be a  locally convex Hausdorff space (=lcHs). A natural problem is to find conditions on the sequence \eqref{seq-1} and the space $E$ such that the tensorized  map
\begin{equation}
\label{tens-map}
Q \varepsilon \operatorname{id}_E: G \varepsilon E \to H \varepsilon E
\end{equation}
is surjective, where we use the Schwartz $\varepsilon$-product \cite{K-Ultradistributions3,S-TheorieDistValeurVect}. By using the $\varepsilon$-product representation of vector-valued function spaces \cite{BFJ, Kruse1}, such conditions may be used to study vector-valued analogs of various classical problems in analysis, e.g., interpolation problems for vector-valued analytic functions \cite{B-D-V-InterpolVVRealAnalFunc,B-LinTopStructClosedIdealsFAlg, Valdivia} and the solvability of PDE on spaces of vector-valued smooth functions \cite{DeKa22, Petzsche, Vogt1983-2} (which is an abstract formulation of the problem of parameter dependence of solutions of PDE \cite{B-D-ParamDepSolDiffEqSpDistSplittingSES,B-D-SplitExactSeqPLSSmoothDepSolLinPDE}).

From now on we suppose that the sequence \eqref{seq-1} consists of Fr\'echet spaces. If $E$ is a Fr\'echet space, it follows from standard results concerning completed tensor products of Fr\'echet spaces that the map \eqref{tens-map} is surjective if either $G$ or $E$ is nuclear (cf.\ Section \ref{sect:Eidelheit}).
The problem is much more difficult if $E$ is a lcHs with a fundamental sequence of bounded sets. If $G$ is Montel and $E$ is reflexive or $E$ is Montel, the  map \eqref{tens-map} is surjective if and only if the map
$$
L(E',G) \to L(E',H), \, T \mapsto Q \circ T
$$
is so. A sufficient condition for the surjectivity of the latter map is that $\Ext^{1}( E',F) = 0$, where  $\Ext^{1}(E', \, \cdot \,) = 0$ is the derived functor of the functor $L(E', \, \cdot \,)$ \cite[Chapter 5]{W-DerivFunctFuncAnal}. 
In \cite{V-FuncExt1Frechet} Vogt gave both sufficient and necessary conditions on pairs of Fr\'echet spaces ($X$,$Y$) such that $\Ext^{1}(X,Y) = 0$ in the following four standard cases: $(i)$ $X \cong \lambda^1(B)$; $(ii)$ $X$ is nuclear; $(iii)$ $Y \cong \lambda^\infty(B)$; $(iv)$ $Y$ is nuclear. He particularly showed that, in these cases, $\Ext^{1}(X,Y) = 0$ if $X$ satisfies  the condition $(\DN)$ and $Y$ satisfies the condition $(\Omega)$ \cite{M-V-IntroFuncAnal}  (see \cite{V-CharUnterraums,V-W-CharQuotientsMartineau} for earlier work in this direction). Later on, full characterizations of  $\Ext^{1}(X,Y) = 0$ were given in the four standard cases \cite[Section 5.2]{W-DerivFunctFuncAnal}. 

When studying the surjectivity of the map \eqref{tens-map}, it seems unnatural to impose additional assumptions on $E$ as is done above. However, without such assumptions on $E$, it is not possible to use results concerning the vanishing of $\Ext^{1}$ for Fr\'echet spaces to conclude the surjectivity of the map \eqref{tens-map}. In \cite{V-TensorFundDFRaumFortsetz} Vogt studied the surjectivity of the map \ref{tens-map} for general locally complete lcHs $E$ that have a fundamental sequence of bounded sets\footnote{In fact, Vogt assumed $E$ to be complete and worked with the $\widehat{\otimes}_\varepsilon$-product, but his results with the exact same proofs hold true in the present setting.} (see \cite{Petzsche,Varol} for related works). 
 If the sequence \eqref{seq-1} consists of Fr\'echet-Schwartz spaces (which in practice is very often the case), the dual sequence
$$
\SES{H'}{G'}{F'}{Q^t}{} 	
$$
consists of $(LB)$-spaces and is  exact, and the  map \eqref{tens-map} is surjective if and only if the map
$$
L(H',E) \to L(G',E), \, T \mapsto  T \circ Q^t
$$
is so. Motivated by this observation, Vogt studied the following general extension problem for $(LB)$-spaces.
\begin{problem}\label{problem1}\emph{
Consider an exact sequence of $(LB)$-spaces 
\begin{equation}
				\label{ses-intro}	
				 \SES{X}{Y}{Z}{\iota}{} 
				 \end{equation}
and let $E$ be another lcHs. Find conditions on the sequence \eqref{ses-intro} and the space $E$ such that the map
\begin{equation}
				\label{map-intro}	
 L(Y,E) \to L(X, E), \, T \mapsto T \circ \iota
\end{equation}
is surjective.}
\end{problem}
\noindent Vogt  mainly studied Problem \ref{problem1} in the setting of K\"othe co-echelon sequence spaces. His main result \cite[Satz 3.9]{V-TensorFundDFRaumFortsetz} asserts that if $Z$ is isomorphic to a closed subspace of a $k^1(A)$-space whose strong dual satisfies $(\Omega)$ and $E$ satisfies the condition $(\condA)$ \cite{V-VektorDistrRandHolomorpherFunk}, then the map \eqref{map-intro} is surjective for each exact sequence \eqref{ses-intro}.  This implies that, given an exact sequence \eqref{seq-1} of Fr\'echet-Schwartz spaces,  the map \eqref{tens-map} is surjective if $F$ is isomorphic to a quotient of  a $\lambda_0(B)$-space that satisfies $(\Omega)$ and $E$ satisfies $(\condA)$ \cite[Satz 4.2]{V-TensorFundDFRaumFortsetz}.

The main objective of this article is to generalize the work \cite{V-TensorFundDFRaumFortsetz} of Vogt by establishing a full characterization of the pairs ($E$, $Z$), $E$ a locally complete lcHs with a fundamental sequence of bounded sets and $Z$ an $(LB)$-space, such that the map \eqref{map-intro} is surjective for each exact sequence \eqref{ses-intro} of $(LB)$-spaces in the following four cases: $(i)$ $E \cong k^{\infty}(A)$; $(ii)$ $E$ is an $(LN)$-space; $(iii)$ $Z \cong k^1(B)$; $(iv)$ $Z$ is an $(LN)$-space. This will be done in Section \ref{sect:main}. Our characterizing condition is inspired by and closely related to the condition $(S_3^*)$ \cite[Definition 5.2.1]{W-DerivFunctFuncAnal} for pairs of Fr\'echet spaces that characterizes the vanishing of  $\Ext^{1}$ \cite[Section 5.2]{W-DerivFunctFuncAnal}.  In \cite{V-TensorFundDFRaumFortsetz} Vogt used a  direct Mittag-Leffler procedure to show his results. Our arguments are very similar in vein, but we use the machinery of the derived projective limit functor \cite{W-DerivFunctFuncAnal}, introduced by Palamodov \cite{P-HomMethTheoryLCS} into the theory of locally convex spaces, to abstract the Mittag-Leffler procedure (as is done in \cite[Section 5.2]{W-DerivFunctFuncAnal} for the vanishing of  $\Ext^{1}$ for Fr\'echet spaces). More precisely, we show in Subsection \ref{subsect-1} that under mild assumptions on $E$ and $Z$ (in particular, in the four above cases), the map \eqref{map-intro} is surjective if and only if
	\begin{equation} \label{eq:IntroProj1}
		\Proj^1 (L(Z_n,E))_{n \in \N} = 0,
	\end{equation}
where $(Z_n)_{n \in \N}$ is an inductive spectrum of Banach spaces whose inductive limit is $Z$. Hereafter, in Subsection \ref{subsect-2}, we evaluate when $\Proj^1 (L(Z_n, E))_{n \in \N} = 0$ in the four above cases by, roughly speaking, dualizing the techniques used in \cite{V-FuncExt1Frechet}.

From the point of view of applications, it is useful to have separate conditions on $E$ and $Z$  ensuring that the map \eqref{map-intro} is surjective for a given exact sequence \eqref{ses-intro} of $(LB)$-spaces.  In Section \ref{sect:sepcond} we show that the map  \eqref{map-intro} is surjective if  $E$ satisfies $(\condA)$ and the strong dual of $Z$ satisfies $(\Omega)$ in the four above mentioned cases. 

In Sections \ref{sect:splitting}--\ref{sect:Eidelheit} we present some applications of our extension results obtained in Sections \ref{sect:main} and \ref{sect:sepcond}.
 Extension problems for continuous linear maps are intrinsically linked to the splitting of short exact sequences, see e.g.\ \cite[Proposition 5.1.3]{W-DerivFunctFuncAnal}. In  Section \ref{sect:splitting}, we discuss the splitting of short exact sequences of $(LB)$-spaces. We consider the question of surjectivity of tensorized maps between Fr\'echet-Schwartz spaces in Section \ref{sec:ExtSESFrechet}. Namely, given an exact sequence \eqref{seq-1} of Fr\'echet-Schwartz spaces, we obtain there sufficient conditions on the pair $(E,F)$ such that the map \eqref{tens-map} is surjective in the following four cases: $(i)$ $E \cong k^{\infty}(A)$; $(ii)$ $E$ is an $(LN)$-space; $(iii)$ $F \cong \lambda_0(B)$; $(iv)$ $F$ is nuclear. We show particularly that, in these cases, the map \eqref{tens-map} is surjective if $F$ satisfies $(\Omega)$ and $E$ satisfies $(\condA)$. As an example of how the latter result may be used to study vector-valued analytic problems, we introduce and investigate the notion of vector-valued Eidelheit families \cite{E-TheorieSystemeLinGleichungen} in Section \ref{sect:Eidelheit}. As simple but instructive concrete applications, we study the vector-valued Borel problem and an interpolation problem for vector-valued holomorphic functions of one complex variable.

In \cite{B-D-ParamDepSolDiffEqSpDistSplittingSES, B-D-SplitExactSeqPLSSmoothDepSolLinPDE}  Bonet and Doma\'nski  initiated the study of the splitting of exact sequence of $(PLS)$-spaces. They obtained splitting results, similar to those for Fr\'echet spaces, for  exact sequences \eqref{seq-1}  of $(PLS)$-spaces  in the case that $H$ is a $(FS)$-space \cite{B-D-ParamDepSolDiffEqSpDistSplittingSES} or $H$ is an $(LS)$-space \cite{B-D-SplitExactSeqPLSSmoothDepSolLinPDE}. In particular,  \cite[Theorem 3.1]{B-D-SplitExactSeqPLSSmoothDepSolLinPDE}  comprises our splittings results from Section \ref{sect:splitting} for exact sequences of  $(LS)$-spaces (and thus also implies our extension results from Section \ref{sect:main} for $(LS)$-spaces via the connection between extension and splitting; see Section \ref{sect:splitting}). The latter also follows by dualizing and applying known splitting results for Fr\'echet spaces. In fact, when restricted to $(LS)$-spaces, the method of Bonet and Doma\'nski precisely boils down to this. This approach does not work for general $(LB)$-spaces, as opposed to ours. It seems interesting to investigate whether the results from  \cite{B-D-ParamDepSolDiffEqSpDistSplittingSES, B-D-SplitExactSeqPLSSmoothDepSolLinPDE} may be extended to general $(PLB)$-spaces via our techniques.

Finally, we would like to emphasize that our work should be seen as a completion of Vogt's work \cite{V-TensorFundDFRaumFortsetz} and all the essential ideas are due to him. We put his results into the framework of the derived projective limit functor and, inspired by \cite{V-FuncExt1Frechet}, extended them in several directions.

\section{Preliminaries}
In this preliminary section, we fix the notation and collect some background material concerning inductive and projective spectra of locally convex Hausdorff spaces that will be used throughout the article. We mainly follow the book \cite{W-DerivFunctFuncAnal}, see also \cite{P-HomMethTheoryLCS, V-TopicsProjSpectraLBSp, V-RegPropLFSp}.
\subsection{Notation}
Let $X$ be a lcHs (= locally convex Hausdorff space). We denote by $X^{\prime}$ the dual of $X$.  Unless specified otherwise, we endow $X^{\prime}$ with the strong topology. Given an absolutely convex bounded subset  $B$ of $X$,  we denote by $X_{B}$ the subspace of $X$ spanned by $B$ endowed with the topology generated by the Minkowski functional of $B$. Since $X$ is Hausdorff, $X_{B}$ is normed. The set $B$ is called a \emph{Banach disk} if $X_B$ is complete. $X$ is said to be \emph{locally complete} if every bounded subset of $X$  is contained in a Banach disk. See \cite[Proposition 5.1.6]{BP} for various characterizations of this condition. Every sequentially complete lcHs is locally complete by \cite[Corollary 23.14]{M-V-IntroFuncAnal}.

 A sequence 
\[ \SES{X}{Y}{Z}{\iota}{Q}  \]
of lcHS is said to be \emph{exact} if $\iota$ and $Q$ are continuous linear maps and the sequence is algebraically exact.
The sequence is called \emph{topologically exact} if $\iota$ is a topological embedding and  $Q$ is open. The sequence \emph{splits} if  $Q$ admits a continuous linear right inverse, or equivalently, if $\iota$ admits a continuous linear left inverse.

Given lcHs $X$ and $Y$,  we denote by $L(X, Y)$ the space consisting of all continuous linear maps from $X$ to $Y$. We define the $\varepsilon$-product \cite{K-Ultradistributions3,S-TheorieDistValeurVect} of $X$ and $Y$ as
	\[ X \varepsilon Y = L(X^{\prime}_{c}, Y) , \]
where the subscript $c$ indicates that we endow $X^{\prime}$ with the topology of uniform convergence on absolutely convex compact subsets of $X$. If $X$ is Montel, then $X \varepsilon Y = L(X^{\prime}, Y)$. The spaces $X \varepsilon Y$ and $Y \varepsilon X$ are canonically isomorphic via transposition \cite[p.~657]{K-Ultradistributions3}.  
Let $X_1,X_2, Y_1, Y_2$ be lcHs. For $T \in L(X_1,X_2)$ and $S \in L(Y_1,Y_2)$ we define the map
\[ T \varepsilon S : X_1 \varepsilon Y_1 \rightarrow X_2 \varepsilon Y_2 , \, R \mapsto   S \circ R \circ T^{t}. \]

\subsection{Inductive spectra}  A sequence $\mathscr{X} = (X_{n})_{n \in \N}$ of lcHs with $X_n \subseteq X_{n+1}$ and continuous inclusion maps is called an \emph{inductive spectrum of lcHs}. We define the \emph{inductive limit} of  $\mathscr{X}$ as the set  $X =\bigcup_{n \in \N} X_n$ endowed with the finest locally convex topology such that all inclusion maps $X_n \to X$ are continuous.  We write $X = \Ind \mathscr{X}$.  
Two inductive spectra of lcHs $(X_{n})_{n \in \N}$ and $(Y_{n})_{n \in \N}$ are said to be \emph{equivalent} if  there are sequences of natural numbers $(k_{n})_{n \in \N}$ and $(l_{n})_{n \in \N}$ with $n \leq k_{n} \leq l_{n} \leq k_{n + 1}$  such that $X_{k_{n}} \subseteq Y_{l_{n}} \subseteq X_{k_{n+1}}$ and the inclusion maps are continuous.

 Let $\mathscr{X} = (X_{n})_{n \in \N}$ be an inductive spectrum of lcHs and set $X = \Ind \mathscr{X}$. The \emph{injective resolution} of $\mathscr{X}$ is defined as the exact sequence 
	\begin{equation}
		\label{eq:InjectiveResolution}
		\SES{\bigoplus_{n \in \N} X_{n}}{\bigoplus_{n \in \N} X_{n}}{X}{d}{\sigma}
	\end{equation} 
where $d((x_{n})_{n \in \N}) = (x_{n} -x_{n - 1})_{n \in \N}$ ($x_0 = 0$) and $\sigma((x_{n})_{n \in \N}) = \sum_{n \in \N} x_{n}$. The map $\sigma$ is open.  $\mathscr{X}$ is called \emph{(weakly) acyclic} if $d$ is a (weak) topological embedding.  

A lcHs $X$ is called an \emph{$(LB)$-space} if $X =  \Ind \mathscr{X}$ for some inductive spectrum of Banach spaces $\mathscr{X}$. We emphasize that we impose $(LB)$-spaces to be Hausdorff (in general the inductive limit of an inductive spectrum of Banach spaces does not need to be Hausdorff, see e.g.\ \cite[Example 24.9]{M-V-IntroFuncAnal}). All inductive spectra of Banach spaces defining a given $(LB)$-space are equivalent, as follows from Grothendieck's factorization theorem \cite[Theorem 24.33]{M-V-IntroFuncAnal}.  
Every $(LB)$-space $X$ has a fundamental sequence of bounded sets, i.e., a sequence of bounded sets $(B_N)_{N \in \N}$ such that for each bounded set $B \subseteq X$ there is $N \in \N$ such that $B \subseteq B_N$ \cite[Corollary 8.3.19]{BP}. Moreover, $(LB)$-spaces are webbed and ultrabornological  \cite[Remark 24.36]{M-V-IntroFuncAnal}, hence we may apply De Wilde's open mapping and closed graph theorem \cite[Theorems 24.30 and 24.31]{M-V-IntroFuncAnal} to maps between them.

An $(LB)$-space $X =  \Ind \mathscr{X}$, with $\mathscr{X}$ an inductive spectrum of Banach spaces, is called (weakly) acyclic if  $\mathscr{X}$ is (weakly) acyclic.
This definition does not depend on the choice of the inductive spectrum of Banach spaces $\mathscr{X}$ defining $X$ \cite[Corollary 1.3]{V-RegPropLFSp}.  We refer to  \cite[Theorem 6.4]{W-DerivFunctFuncAnal}  for various characterizations of this notion.

 A lcHs $X$ is called an \emph{$(LN)$-space} if $X =  \Ind \mathscr{X}$ for some inductive spectrum of Banach spaces $\mathscr{X} = (X_{n})_{n \in \N}$ with nuclear inclusion maps $X_n \to X_{n+1}$. A lcHs $X$ is an $(LN)$-space if and only if it is a nuclear  $(LB)$-space. Every $(LN)$-space is acyclic (which e.g.\ follows from \cite[Theorem 6.4]{W-DerivFunctFuncAnal} and the fact that nuclear maps between Banach spaces are compact).


\subsection{Projective spectra}
A \emph{projective spectrum} $\mathscr{X}$ is a sequence of vector spaces $(X_{n})_{n \in \N}$ together with linear maps $\varrho_{m}^{n} : X_{m} \rightarrow X_{n}$, $n \leq m$, such that $\varrho^{n}_{n} = \operatorname{id}_{X_n}$ and $\varrho_{k}^{n} = \varrho_{m}^{n} \circ \varrho_{k}^{m}$ for $n \leq m \leq k$.   We write $\mathscr{X} = (X_{n}, \varrho_{m}^{n})$. Consider the linear map
	\[ \Psi : \prod_{n \in \N} X_{n} \rightarrow \prod_{n \in \N} X_{n}, \, (x_{n})_{n \in \N} \mapsto (x_{n} - \varrho_{n + 1}^{n}(x_{n + 1}))_{n \in \N} . \]
We define the \emph{projective limit} of $\mathscr{X}$ as
	\[ \Proj \mathscr{X}  = \ker \Psi =  \left\{ (x_{n})_{n \in \N} \in \prod_{n \in \N} X_{n} \mid \varrho_{m}^{n}(x_{m}) = x_{n} \text{ for all } n \leq m \right\} \]
and the \emph{first derived projective limit} of $\mathscr{X}$ as
	\[ \Proj^{1} \mathscr{X} = \left( \prod_{n \in \N} X_{n} \right) / \im \Psi . \]

Let $\mathscr{X} = (X_{n}, \varrho_{m}^{n})$ and $\mathscr{Y} = (Y_{n}, \sigma_{m}^{n})$ be two projective spectra.  A  \emph{morphism} $f = (f_{n})_{n \in \N}: \mathscr{X}  \to \mathscr{Y}$ consists of  linear maps $f_{n}: X_{n} \to Y_{n}$ such that $f_{n} \circ \varrho_{m}^{n} = \sigma_{m}^{n} \circ f_{m}$ for $n \leq m$. The kernel of $f$ is defined as the projective spectrum consisting of the sequence of vector spaces $(\ker f_n)_{n \in \N}$ and spectral maps  ${\varrho_{m}^{n}}_{\mid \ker f_m}: \ker f_m \to \ker f_n$.  We define the linear map
$$ 
\Proj f : \Proj \mathscr{X} \rightarrow \Proj \mathscr{Y}, \, (x_{n})_{n \in \N} \mapsto (f_{n}(x_{n}))_{n \in \N}.
$$

Two projective spectra $\mathscr{X} = (X_{n}, \varrho_{m}^{n})$ and $\mathscr{Y} = (Y_{n}, \sigma_{m}^{n})$ are called \emph{equivalent} if there are sequences $(k_{n})_{n \in \N}$ and $(l_{n})_{n \in \N}$ of natural numbers such that $n \leq l_{n} \leq k_{n} \leq l_{n + 1}$ and linear maps $\alpha_{n} : X_{k_{n}} \rightarrow Y_{l_{n}}$ and $\beta_{n + 1} : Y_{l_{n + 1}} \rightarrow X_{k_{n}}$ such that $\beta_{n + 1} \circ \alpha_{n + 1} = \varrho_{k_{n + 1}}^{k_{n}}$ and $\alpha_{n} \circ \beta_{n + 1} = \sigma_{l_{n + 1}}^{l_{n}}$. In such a case, we have that $\Proj \mathscr{X} \cong \Proj \mathscr{Y}$ and $\Proj^{1} \mathscr{X} \cong \Proj^{1} \mathscr{Y}$ \cite[Proposition 3.1.7]{W-DerivFunctFuncAnal}.

We shall use the following fundamental property of the derived projective limit. 

	\begin{proposition}\cite[Proposition 3.1.8]{W-DerivFunctFuncAnal}
		\label{t:AbstractMittagLeffler}
		Let $\mathscr{Y} = (Y_{n}, \varrho_{m}^{n})$ and $\mathscr{Z} = (Z_{n}, \tau_{m}^{n})$ be two projective spectra and let $f = (f_{n})_{n \in \N} : \mathscr{Y} \rightarrow \mathscr{Z}$ be a morphism. Suppose that for all $n \in \N$ there is $m \in \N$ such that  
					$ \tau^{n}_{m}(Z_{m}) \subseteq f_{n}(Y_{n})$. Let  $\mathscr{X}$ be the kernel of $f$. If $\Proj^{1} \mathscr{X} = 0$, then $\Proj f : \Proj \mathscr{Y} \rightarrow \Proj \mathscr{Z}$ is surjective.
	\end{proposition}


For projective spectra  $\mathscr{X} = (X_{n}, \varrho^{n}_{m})$  consisting of $(LB)$-spaces $X_n$ and continuous linear maps $\varrho^{n}_{m}$ there are several necessary and sufficient conditions for $\Proj^{1}\mathscr{X} = 0$ to hold \cite{P-HomMethTheoryLCS,V-TopicsProjSpectraLBSp, W-DerivFunctFuncAnal}. We will employ the following ones.
\begin{proposition}
\label{P:CondProj1=0} Let $\mathscr{X} = (X_{n}, \varrho^{n}_{m})$ be a projective spectrum consisting of $(LB)$-spaces $X_n$ and continuous linear maps $\varrho^{n}_{m}$.  Suppose that $X_n = \Ind \mathscr{X}_n$ with  $\mathscr{X}_n= (X_{n,N})_{N \in \N}$ an inductive spectrum of Banach spaces. Let $B_{n, N}$ be the unit ball in $X_{n, N}$. Consider the conditions
			\begin{equation}
				\label{eq:SuffCondProj1=0X}
				\begin{gathered}
					\forall n \in \N ~ \exists m \geq n ~ \forall k \geq m ~ \exists N \in \N ~ \forall M \geq N, \varepsilon > 0 ~ \exists K \geq M, C > 0 \, : \\
					\varrho^{n}_{m}(B_{m,M}) \subseteq C \varrho^{n}_{k}(B_{k,K})  + \varepsilon B_{n,N},
					\end{gathered}
			\end{equation}
		and
			\begin{equation}
				\label{eq:NeccCondProj1=0X}
				\begin{gathered}
					\forall n \in \N ~ \exists m \geq n ~ \forall k \geq m ~ \exists N \in \N ~ \forall M \geq N~ \exists K \geq M, C > 0  \, : \\
					\varrho^{n}_{m}(B_{m,M}) \subseteq C (\varrho^{n}_{k}(B_{k,K})  + B_{n,N}).
					\end{gathered}
			\end{equation}
		Then, \eqref{eq:SuffCondProj1=0X} $\Rightarrow$ $\Proj^{1}\mathscr{X} = 0$ $\Rightarrow$ \eqref{eq:NeccCondProj1=0X}.
\end{proposition}
\begin{proof}
By Grothendieck's factorization theorem, every Banach disk in $X_n$ is contained in $CB_{n, N}$ for some $C>0$ and $N \in \N$. Hence, \eqref{eq:SuffCondProj1=0X} $\Rightarrow$ $\Proj^{1} \mathscr{X}  = 0$  follows from \cite[Theorem 3.2.14]{W-DerivFunctFuncAnal}.
The implication $\Proj^{1} \mathscr{X} = 0$ $\Rightarrow$ \eqref{eq:NeccCondProj1=0X} is shown in \cite[Lemma 2.3]{V-TopicsProjSpectraLBSp}.
\end{proof}

\subsection{Sequence spaces}\label{ss:seq}
Throughout this article, $I$ will always denote an arbitrary index set. We define the Banach spaces
	\begin{gather*}\ell^{1}(I) = \{ c = (c_{j})_{j \in I} \in \C^{I} \mid \|c\|_{\ell^{1}(I)} = \sum_{j \in I} |c_{j}|< \infty \} , \\
	\ell^{\infty}(I) = \{ c = (c_{j})_{j \in I} \in \C^{I} \mid \|c\|_{\ell^{\infty}(I)} = \sup_{j \in I} |c_{j}| < \infty \}.
	\end{gather*} 
We will often drop the index set $I$ from the notation and simply write $\ell^{1} = \ell^{1}(I)$ and $\ell^{\infty} = \ell^{\infty}(I)$; we will do the same for other spaces as well.
For a sequence $a = (a_{j})_{j \in I}$ of positive real numbers we define the Banach spaces
	\begin{align*} 
		 \ell^{p}(a)  & = \{ c = (c_{j})_{j \in I} \in \C^{I} \mid \|c\|_{\ell^{p}(a)} = \|ac\|_{\ell^{p}} < \infty \} , \qquad p =1,\infty.
	\end{align*}
We write $c_{0}(a)$ for the space consisting of all $c \in \ell^{\infty}(a)$ such that for every $\varepsilon > 0$ there is a finite subset $I_{\varepsilon} \subseteq I$ such that $\sup_{j \in I \setminus I_{\varepsilon}} |a_jc_{j}| \leq \varepsilon$. Then $c_{0}(a)$  is a closed subspace of $\ell^{\infty}(a)$ and hence also a Banach space.
	
Let $A = (a_{n})_{n \in \N} =  (a_{n,j})_{n \in \N, j \in I} $ be a sequence of  positive sequences on $I$ such that $a_{n + 1, j} \leq a_{n, j}$ for all $n \in \N$ and $j \in I$.  We define the $(LB)$-spaces
	\[ k^{p}(A) = \Ind (\ell^{p}(a_{n}))_{n \in \N} , \quad p =1, \infty, \qquad k_{0}(A) = \Ind (c_{0}(a_{n}))_{n \in \N} . \]
Let $B = (b_{n})_{n \in \N} =  (b_{n,j})_{n \in \N, j \in I} $ be a sequence of  positive sequences on $I$ such that $b_{n, j} \leq b_{n+1, j}$ for all $n \in \N$ and $j \in I$. We define the spaces
\[ \lambda^{p}(B) = \bigcap_{n \in \N} \ell^{p}(b_{n}) , \quad p =1, \infty, \qquad \lambda_{0}(B) = \bigcap_{n \in \N} c_{0}(b_{n}), \]
and endow them with their natural Fr\'echet space topology.

	
\section{An extension result for $(LB)$-spaces}\label{sect:main}

This section is devoted to Problem \ref{problem1} from the introduction. Our main result will be the characterization of the surjectivity of \eqref{map-intro} in terms of $E$ and $Z$ in the standard cases, see Theorem \ref{t:MainExtensionThm}. 
Our characterization will be in the form of condition $(\condS)$ which we introduce and study in the ensuing Subsection \ref{sec:condS}. 
Similar to Vogt's work \cite{V-TensorFundDFRaumFortsetz}, the proof of Theorem \ref{t:MainExtensionThm} will consist of two steps. 
First, in Subsection \ref{subsect-1}, we show that under mild assumptions on $E$ and $Z$  the surjectivity of \eqref{map-intro} is equivalent to $\Proj^{1} L(Z, E) = 0$, see Theorem \ref{t:MittagLeffler}. This shows that we may solve our extension problem using the abstract Mittag-Leffler procedure. This subsection corresponds to Section 2 of \cite{V-TensorFundDFRaumFortsetz}, in particular Satz 2.3. 
Then, in Subsection \ref{subsect-2}, we characterize when $\Proj^{1} L(Z, E) = 0$, under certain assumptions on $E$ and $Z$, by using Proposition \ref{P:CondProj1=0} and  by essentially dualizing the arguments  from \cite{V-FuncExt1Frechet}. This is similar to what is done in Section 3 of \cite{V-TensorFundDFRaumFortsetz}, specifically Satz 3.9.


\subsection{The conditions $(\condS)$ and $(\condWS)$}
\label{sec:condS}

The following conditions will play an essential role in the rest of this article.
	
\begin{definition} Let $E$ be a lcHs with a fundamental increasing sequence of bounded sets $(B_{N})_{N \in \N}$ and let $\mathscr{Z} = (Z_n)_{n\in\N}$ be an inductive spectrum of Banach spaces.  
\begin{itemize}
\item[(i)] The pair  $(E, \mathscr{Z})$ is said to satisfy $(\condS)$ if		
					\begin{equation*}
						\label{eq:S}
						\begin{gathered}
							\forall n \in \N ~ \exists m \geq n ~ \forall k \geq m ~ \exists N \in \N ~ \forall M \geq N, \varepsilon >0~ \exists K \geq M, C > 0 \,: \\
							 \|z\|_{Z_{m}} B_{M} \subseteq C \|z\|_{Z_{k}} B_{K} + \varepsilon \|z\|_{Z_{n}} B_{N} , \qquad \forall  z \in Z_n.
						\end{gathered}
					\end{equation*}	

\item[(ii)]The pair  $(E, \mathscr{Z})$ is said to satisfy $(\condWS)$ if
					\begin{equation*}
						\label{eq:WS}
						\begin{gathered}
					 		\forall n \in \N ~ \exists m \geq n ~ \forall k \geq m ~ \exists N \in \N ~ \forall M \geq N ~ \exists K \geq M, C > 0 \,: \\
							 \|z\|_{Z_{m}} B_{M} \subseteq C \left( \|z\|_{Z_{k}} B_{K} + \|z\|_{Z_{n}} B_{N} \right), \qquad \forall  z \in Z_n.
						\end{gathered}
					\end{equation*}	
\end{itemize}
\end{definition}
\noindent The conditions $(\condS)$ and  $(\condWS)$ do not depend on the choice of the fundamental increasing sequence of bounded sets $(B_{N})_{N \in \N}$ in $E$. Moreover, if  $\mathscr{Z}$ and  $\widetilde{\mathscr{Z}}$ are two equivalent inductive spectra of Banach spaces, then  $(E,\mathscr{Z})$ satisfies $(\condS)$ or $(\condWS)$  if and only if $(E,\widetilde{\mathscr{Z}})$ does so.
Let $Z = \Ind \mathscr{Z}$ be an $(LB)$-space, with  $\mathscr{Z}$ an inductive spectrum of Banach spaces. We say that $(E, Z)$ satisfies $(\condS)$  or $(\condWS)$  if $(E,  \mathscr{Z})$ does so. Since all inductive spectra of Banach spaces defining $Z$ are equivalent, this definition does not depend on the choice of the spectrum $ \mathscr{Z}$ of Banach spaces defining $Z$.

We now study the connection between $(\condS)$ and $(\condWS)$. To this end, we will use that an inductive spectrum of Banach spaces $\mathscr{Z} = (Z_n)_{n\in\N}$ is acyclic if and only if it satisfies condition $(\condQ)$ \cite[Theorem 2.7]{W-AcyclicIndSpectraFrechetSp}:
			\begin{equation*}
				\begin{gathered}
				\forall n \in \N ~ \exists m \geq n ~ \forall k \geq m, \varepsilon > 0 ~ \exists C > 0 \, : \\
				  \|z\|_{Z_{m}} \leq C \|z\|_{Z_{k}} + \varepsilon \|z\|_{Z_{n}}, \qquad  \forall z \in Z_{n}.
				  		\end{gathered}
			\end{equation*}

	\begin{proposition}
		\label{l:S=WS+acyclic}
		Let $E$ be a  lcHs with a fundamental increasing sequence of bounded sets $(B_{N})_{N \in \N}$  and let $\mathscr{Z} = (Z_n)_{n\in\N}$ be an inductive spectrum of Banach spaces.  Then, $(E, \mathscr{Z})$ satisfies $(\condS)$ if and only if $(E, \mathscr{Z} )$ satisfies $(\condWS)$ and $\mathscr{Z}$ is acyclic.
	\end{proposition}
	
	\begin{proof}
	Suppose first that $(E, \mathscr{Z} )$ satisfies $(\condS)$. Obviously, $(E, \mathscr{Z} )$ satisfies $(\condWS)$. We now show that $\mathscr{Z}$ satisfies $(\condQ)$ and thus is acyclic. We may assume that each $B_{N}$ is absolutely convex, closed, and contains a non-zero element. Let $n \in \N$ be arbitrary. Choose $m \geq n$ according to $(\condS)$. Let $k \in \N$ be arbitrary. Choose $N \in \N$ according to $(\condS)$. Let $\varepsilon > 0$ be arbitrary. Condition $(\condS)$ yields that  there are $K \geq N$ and $C > 0$ such that 
$$
\|z\|_{Z_{m}} B_{N} \subseteq C \|z\|_{Z_{k}} B_{K} + \frac{\varepsilon}{2} \|z\|_{Z_{n}} B_{N} , \qquad \forall  z \in Z_n.
$$
 Fix $e \in B_N$ such that $2e \notin B_N$. By the bipolar theorem, there is $e' \in 2B^\circ_N$ such that $|\langle e',e\rangle| \geq 1$. Hence, for all $z \in Z_n$ 
\begin{eqnarray*}
\|z\|_{Z_{m}} \leq  |\langle e',\|z\|_{Z_{m}}e\rangle| &\leq&  C\sup_{e_1 \in B_K}|\langle e',e_1 \rangle | \|z\|_{Z_{k}}  + \frac{\varepsilon}{2}\sup_{e_2 \in B_N}|\langle e',e_2 \rangle |  \|z\|_{Z_{n}}  \\
&\leq& C' \|z\|_{Z_{k}}  + \varepsilon  \|z\|_{Z_{n}},
\end{eqnarray*}
 where $C' = C\sup_{e_1 \in B_K}|\langle e',e_1 \rangle|$. Conversely, suppose that $(E, \mathscr{Z})$ satisfies $(\condWS)$ and that $\mathscr{Z}$ is acyclic and thus satisfies $(\condQ)$. 
 Let $n \in \N$ be arbitrary. Choose $n' \geq n$ according to  $(\condQ)$ and $m \geq n'$ according to $(\condWS)$. Let $k \geq m$ be arbitrary. Choose  $N \in \N$ according to $(\condWS)$. Let $M \geq N$ and $\varepsilon > 0$ be arbitrary. Condition  $(\condWS)$ implies that there is $K \geq M$ and $C >0$ such that
 $$
 \|z\|_{Z_{m}} B_{M} \subseteq C( \|z\|_{Z_{k}} B_{K} + \|z\|_{Z_{n'}} B_{N} ), \qquad \forall  z \in Z_{n'}.
$$
 By  $(\condQ)$, there is $C' > 0$ such that
 $$
 \|z\|_{Z_{n'}} \leq C' \|z\|_{Z_{k}} + \frac{\varepsilon}{C} \|z\|_{Z_{n}}, \qquad  \forall z \in Z_{n}.
 $$
 Hence,
 $$
 \|z\|_{Z_{m}} B_{M} \subseteq C(1 +C')\|z\|_{Z_{k}} B_{K} + \varepsilon\|z\|_{Z_{n}} B_{N}, \qquad  \forall z \in Z_{n}.
$$	\end{proof}
A bounded set $B$ in a lcHs $E$ is called \emph{fundamental} if for every bounded set $D \subseteq E$ there is $C > 0$ such that $D \subseteq CB$. A quasibarrelled lcHs has a fundamental bounded set if and only if it is normable. However, there exist non-normable lcHs with a fundamental bounded set, e.g., infinite-dimensional normed spaces endowed with the weak topology.

\begin{corollary}
		\label{l:S=acyclic}
		Let $E$ be a  lcHs with a fundamental bounded set and let $\mathscr{Z}$ be an inductive spectrum of Banach spaces.  Then, $(E, \mathscr{Z})$ satisfies $(\condS)$ if and only if $\mathscr{Z}$ is acyclic.
	\end{corollary}	
	\begin{proof}
Since $E$ has a fundamental bounded set, $(E, \mathscr{Z})$ clearly satisfies $(\condWS)$. Hence, the result follows from Proposition \ref{l:S=WS+acyclic}. 
	\end{proof}
	

	\begin{corollary}
		\label{c:S=WSIfENotBanach}
		Let $E$ be a lcHs with a  fundamental increasing sequence of bounded sets $(B_{N})_{N \in \N}$  that does not have a fundamental bounded set and let $\mathscr{Z} = (Z_n)_{n\in\N}$ be an inductive spectrum of Banach spaces. Then, $(E, \mathscr{Z})$ satisfies $(\condS)$ if and only if $(E, \mathscr{Z})$ satisfies $(\condWS)$. 
	\end{corollary}
	
	\begin{proof}
		By Proposition \ref{l:S=WS+acyclic}, it suffices to show that $(\condWS)$ implies that $\mathscr{Z}$ satisfies $(\condQ)$ and thus is acyclic. We may suppose that each $B_N$ is  absolutely convex and closed. Let $n \in \N$ be arbitrary. Choose $m \geq n$ according to $(\condWS)$. Let $k \in \N$ and $\varepsilon >0$ be arbitrary. Choose $N \in \N$ according to $(\condWS)$. Since $B_N$ is not a fundamental bounded set in $E$, there is $M \geq N$ such that for all $r>0$ it holds that $B_M \not \subseteq r B_N$. Condition $(\condWS)$ implies that  there are $K \geq M$ and $C > 0$ such that 
$$
\|z\|_{Z_{m}} B_{M} \subseteq C(  \|z\|_{Z_{k}} B_{K} + \|z\|_{Z_{n}} B_{N} ) , \qquad \forall  z \in Z_n.
$$
 Fix $e \in B_M$ such that $e \notin \frac{C}{\varepsilon}B_N$. By the bipolar theorem, there is $e' \in \frac{\varepsilon}{C}B^\circ_N$ such that $|\langle e',e\rangle| \geq 1$. Hence, for all $z \in Z_n$
\begin{eqnarray*}
\|z\|_{Z_{m}} \leq  |\langle e',\|z\|_{Z_{m}}e\rangle| &\leq&  C(\sup_{e_1 \in B_K}|\langle e',e_1 \rangle | \|z\|_{Z_{k}}  + \sup_{e_2 \in B_N}|\langle e',e_2 \rangle |  \|z\|_{Z_{n}})  \\
&\leq& C' \|z\|_{Z_{k}}  + \varepsilon  \|z\|_{Z_{n}},
\end{eqnarray*}
 where $C' = C\sup_{e_1 \in B_K}|\langle e',e_1 \rangle|$.	\end{proof}

\subsection{Statement of the main result}	 We start with the following fundamental definition.

\begin{definition}
	Let $E$ be a lcHs. An $(LB)$-space $Z$ is said to be \emph{weakly $E$-acyclic} if  for every short exact sequence
					\begin{equation}
				\label{ses}	
				 \SES{X}{Y}{Z}{\iota}{Q} 
				 \end{equation}
				of $(LB)$-spaces  it holds that the map
				\[
				\iota_*: L(Y,E) \to L(X, E), \, T \mapsto T \circ \iota
				\]
				is surjective.				
	\end{definition}		
	
	\begin{remark}\label{remark-wa}
	(i) The term weakly $E$-acyclic stems from the fact that an $(LB)$-space $Z$ is weakly acyclic if and only if it is weakly $\C$-acyclic (see Proposition \ref{t:MainExtensionThm2} below). \\
	\noindent (ii)  Let $E$ be a lcHs and let $Z$ be an $(LB)$-space. Then, for every exact sequence \eqref{ses} of $(LB)$-spaces it holds that
	\[\begin{tikzcd}
					0 \arrow{r} & L(Z, E) \arrow{r}{Q_*} & L(Y, E) \arrow{r}{\iota_*} & L(X, E) 
				\end{tikzcd}
			\]
			is exact: the injectivity of ${Q}_*$ is clear, while the exactness at $ L(Y, E)$ follows from the fact that $Q$ is open, which in turn is a consequence of  De Wilde's open mapping theorem. Hence, $Z$ is weakly $E$-acyclic if and only if 
			 for every exact sequence \eqref{ses} of $(LB)$-spaces  it holds that the sequence
			 \[ \SES{L(Z, E)}{L(Y, E)}{L(X, E)}{{Q}_*}{{\iota}_*}  \]
				is exact.
					\end{remark}


We are ready to formulate the main result of this section, its proof is given in the next two subsections.
	\begin{theorem}
		\label{t:MainExtensionThm}
		Let $E$ be a locally complete lcHs with a  fundamental sequence of bounded sets and let $Z$ be an $(LB)$-space. Suppose that one of the following assumptions is satisfied:
			\begin{itemize} 
			\item[(a)] $E \cong k^{\infty}(A)$ and $E$ is non-normable.
			\item[(b)]   $E$  is an infinite-dimensional $(LN)$-space.
			\item[(c)]  $Z \cong k^{1}(A)$.
			\item[(d)]  $Z$ is an $(LN)$-space.
			\end{itemize}
		 Then, $Z$ is weakly $E$-acyclic if and only if $(E, Z)$ satisfies $(\condS)$.
	\end{theorem}
	
\noindent 
\begin{remark}\label{remWSS}
\noindent (i) If  (a), (b), or (d) holds, then $(E, Z)$ satisfies $(\condS)$ if and only if it satisfies $(\condWS)$. For (a) and (b) this follows from Corollary \ref{c:S=WSIfENotBanach}, while for (d) this is a consequence of Proposition \ref{l:S=WS+acyclic} and the fact that every $(LN)$-space is acyclic. \\
\noindent (ii) The reason that we assume $E$ to be non-normable in (a) and $E$ to be infinite-dimensional in (b) is to exclude the case  $E \cong \ell^{\infty}$. These spaces will be treated separately in  Proposition \ref{t:MainExtensionThm2} below.
\end{remark}

\subsection{Weak $E$-acyclicity and the  derived projective limit}\label{subsect-1}	 Let $E$ be a lcHs and let $\mathscr{Z} = (Z_{n})_{n \in \N}$ be an inductive spectrum of Banach spaces. We denote by  $\mathscr{L}(\mathscr{Z}, E)$  the projective spectrum $(L(Z_{n}, E), \varrho_m^n)$ with
			\[ \varrho_m^n: L(Z_m, E) \to L(Z_{n}, E), \, T \mapsto T_{\mid Z_n}.  \]
 Note that 		
\begin{equation}
\label{identification}
L( \Ind \mathscr{Z} , E) \to \Proj \mathscr{L}(\mathscr{Z}, E), \, T \mapsto (T_{\mid Z_n})_{n \in \N}
\end{equation}
is an isomorphism of vector spaces. If  $\mathscr{Z}$ and  $\mathscr{Z}'$ are two equivalent inductive spectra of Banach spaces, then the projective spectra $\mathscr{L}(\mathscr{Z}, E)$ and $\mathscr{L}(\mathscr{Z}', E)$ are also equivalent.
Let $Z = \Ind \mathscr{Z}$ be an $(LB)$-space, with  $\mathscr{Z}$ being an inductive spectrum of Banach spaces.  We write $\Proj^{1} \mathscr{L}(Z, E) = 0$ to indicate that $\Proj^{1} \mathscr{L}(\mathscr{Z}, E) = 0$.
Since all inductive spectra of Banach spaces defining $Z$ are equivalent, this definition does not depend on the choice of the spectrum $ \mathscr{Z}$ of Banach spaces defining $Z$.

Our goal in this subsection is to characterize weak $E$-acyclicity of $Z$ by $\Proj^{1} \mathscr{L}(Z, E) = 0$. We start with the following result.

	\begin{proposition}
		\label{p:EAcyclicImpliesProj1=0}
		Let $E$ be a lcHs and $Z$ an $(LB)$-space. If $Z$ is weakly $E$-acyclic, then $\Proj^{1} \mathscr{L}(Z, E) = 0$.
	\end{proposition}
	
	\begin{proof}
		Let $Z = \Ind \mathscr{Z}$, with  $\mathscr{Z} = (Z_n)_{n \in \N}$  an inductive spectrum of Banach spaces.  We need to show that the map
		$$
		\Psi : \prod_{n \in \N} L(Z_{n},E) \rightarrow \prod_{n \in \N} L(Z_{n},E) , \, (T_{n})_{n \in \N} \mapsto (T_{n} - T_{n+1\mid Z_n})_{n \in \N} 
		$$
		is surjective. Consider the injective resolution \eqref{eq:InjectiveResolution} of $\mathscr{Z}$. Since $Z$ is weakly $E$-acyclic, the map
		$
		d_\ast: L(\bigoplus_{n \in \N} Z_{n}, E) \to L(\bigoplus_{n \in \N} Z_{n}, E)
		$
		is surjective. Let
		$$
		\Phi: \prod_{n \in \N} L(Z_{n},E) \to  L(\bigoplus_{n \in \N} Z_{n}, E), \, (T_{n})_{n \in \N}  \mapsto ( (x_{n})_{n \in \N}  \mapsto \sum_{n \in \N} T_n(x_n))
		$$
		be the natural isomorphism and note that $\Psi = \Phi^{-1} \circ d_* \circ \Phi$. As $d_*$ is surjective,  we find that also  $\Psi$ is surjective.	
			\end{proof}
	
We now look for conditions on $E$ and $Z$ ensuring that the converse of Proposition \ref{p:EAcyclicImpliesProj1=0} also holds true. 
To this end, we introduce the following definition.

	\begin{definition}
		\label{d:locallyEacyclic}
		Let $E$ be a lcHs.
		\begin{itemize}
		\item[(i)] A Banach space $H$ is called \emph{weakly $E$-acyclic in the category of Banach spaces} if for  every exact sequence
				\[ \SES{F}{G}{H}{\iota}{} \]
			of Banach spaces the map $\iota_*: L(G, E) \to L(F, E)$ is surjective.
			
		\item[(ii)] An $(LB)$-space $Z$ is called \emph{locally weakly $E$-acyclic} if $Z = \Ind \mathscr{Z}$ for some inductive spectrum of Banach spaces $\mathscr{Z} = (Z_{n})_{n \in \N}$ such that each inclusion map $i: Z_n \to Z_{n+1}$ factors through a Banach space $H$ that is weakly $E$-acyclic in the category of Banach spaces, i.e.,  that there are $a \in L(Z_n, H)$ and $b \in L(H, Z_{n+1})$ such that $i = b \circ a$.	
		\end{itemize}		
	\end{definition}
	
The main result of this subsection may now be formulated as follows.
	\begin{theorem}
		\label{t:MittagLeffler}
		Let $E$ be a lcHs and let $Z$ be a locally weakly $E$-acyclic $(LB)$-space. Then, $Z$ is weakly $E$-acyclic if and only if $\Proj^{1} \mathscr{L}(Z, E) = 0$.
	\end{theorem}
	

We need the following lemma. 
	\begin{lemma}
		\label{l:StepwiseSES}
		Let 
			\[ \SES{X}{Y}{Z}{\iota}{Q}  \]
			be an exact sequence of $(LB)$-spaces. Let $\mathscr{Y} = (Y_n)_{n \in \N}$ be an inductive spectrum of Banach spaces such that $Y = \Ind \mathscr{Y}$. Then, there exist inductive spectra of Banach spaces $\mathscr{X} = (X_n)_{n \in \N}$ and $\mathscr{Z} =(Z_n)_{n \in \N}$  with $X = \Ind \mathscr{X}$ and $Z = \Ind \mathscr{Z}$ such that  the sequence
			\begin{equation}
			\label{buildingspectra}
			\SES{X_{n}}{Y_{n}}{Z_{n}}{\iota_{\mid X_n}}{Q_{\mid Y_n}} 
			\end{equation}
		is exact for each $n \in \N$.
	\end{lemma}
	
	\begin{proof}
	Let $n \in \N$. We set $X_n = \iota^{-1}( \ker Q \cap Y_n)$ and endow it with the initial topology with respect to the map  $\iota_{\mid X_n}: X_n \to Y_n$. We set $Z_n = Q(Y_n)$ and endow it with the final topology with respect to the map $Q_{\mid Y_n}: Y_n \to Z_n$. Then, $\mathscr{X}  =(X_n)_{n \in \N}$ and  $\mathscr{Z} = (Z_n)_{n \in \N}$ are inductive spectra of Banach spaces with $X = \bigcup_{n \in \N}X_n$ and $Z = \bigcup_{n \in \N}Z_n$ such that the sequence \eqref{buildingspectra} is exact for each $n \in \N$. Moreover, $X = \Ind \mathscr{X} $ and $Z = \Ind \mathscr{Z}$, as follows from  De Wilde's open mapping theorem and the fact that the inclusion maps $X \to \Ind \mathscr{X}$ and  $\Ind \mathscr{Z}  \to Z$ are continuous.
	\end{proof}

	\begin{proof}[Proof of Theorem \ref{t:MittagLeffler}]
		In view of Proposition \ref{p:EAcyclicImpliesProj1=0}, it suffices to prove that  $Z$ is  weakly $E$-acyclic if $\Proj^{1} \mathscr{L}(Z, E) = 0$. Let
			\[ \SES{X}{Y}{Z}{\iota}{Q} \]
			 be an exact sequence of $(LB)$-spaces.  
			 By Lemma \ref{l:StepwiseSES}, we find inductive spectra of Banach spaces $\mathscr{X} = (X_n)_{n \in \N}$, $\mathscr{Y} = (Y_n)_{n \in \N}$, and $\mathscr{Z} =(Z_n)_{n \in \N}$  with $X = \Ind \mathscr{X}$,  $Y = \Ind \mathscr{Y}$, and $Z = \Ind \mathscr{Z}$ such that  the sequence
			\begin{equation}
			\label{buildingspectra-3}
			\SES{X_{n}}{Y_{n}}{Z_{n}}{\iota_n}{Q_n} 
			\end{equation}
		is exact for each $n \in \N$, where $\iota_n = \iota_{\mid X_n}$ and $Q_n = Q_{\mid Y_n}$.  The sequence
			\begin{equation}
			\label{buildingspectra-2}
							\begin{tikzcd}
					0 \arrow{r} & L(Z_{n}, E) \arrow{r}{(Q_n)_*} & L(Y_{n}, E) \arrow{r}{(\iota_n)_*} & L(X_{n}, E) 
				\end{tikzcd}
			\end{equation}
			is exact for each $n \in \N$ (cf.\ Remark \ref{remark-wa}(ii)). Consider the morphism of projective spectra $((\iota_n)_*)_{n \in \N}: \mathscr{L}(\mathscr{Y}, E) \to \mathscr{L}(\mathscr{X}, E)$ and note that $\Proj ((\iota_n)_*)_{n \in \N} = \iota_*$ (under the identification \eqref{identification}).  Moreover, as the sequence \eqref{buildingspectra-2} is exact for each $n \in \N$,  the kernel of $((\iota_n)_*)_{n \in \N}$ is equivalent to  $\mathscr{L}(\mathscr{Z}, E)$. Since $\Proj^1 \mathscr{L}(\mathscr{Z}, E) = 0$, Proposition  \ref{t:AbstractMittagLeffler} yields that it suffices to show that  for every $n \in \N$ there is $m \geq n$ such that
			\begin{equation}
			\label{TP1}
			 \forall T \in L(X_m,E)  ~ \exists S \in L(Y_n,E) : T_{|X_n} = S \circ \iota_n.
			\end{equation}
Let $n \in \N$ be arbitrary.  Since all inductive spectra of Banach spaces defining $Z$ are equivalent, there is  $m \geq n$ such that the inclusion map $i: Z_n \to Z_m$ factors through a Banach space $H$ that is weakly $E$-acyclic in the category of Banach spaces.  Let $a \in L(Z_n, H)$ and $b \in L(H,Z_m)$ be such that $i= b \circ a$. We define the Banach space $G = \{ (y, h) \in Y_{m} \times H \mid Q_{m}(y) = b(h) \}$ and the continuous linear maps $j: X_m \to G, \,  x \mapsto (\iota_m(x),0)$ and $\pi: G \to H, \, (y,h) \mapsto h$. The sequence
\[			
			\SES{X_{m}}{G}{H}{j}{\pi} 
		\]
					is exact because \eqref{buildingspectra-3}  is exact. As $H$ is weakly $E$-acyclic in the category of Banach spaces, the map $j_*: L(G, E) \mapsto L(X_m, E)$ is surjective. 
					Consider the continuous linear map $\tau: Y_n \to G, \, y \mapsto (y,a(Q_n(y)))$, which is well-defined because $i= b \circ a$. The property \eqref{TP1} now follows from the fact that
					$j_{\mid X_{n}}$ can be factored as $j_{\mid X_{n}} = \tau \circ \iota_n$.		
	\end{proof}	
	

We now give some sufficient conditions on a lcHs $E$ and an $(LB)$-space $Z$ such that $Z$ is locally weakly $E$-acyclic. 

Recall that a Banach space $F$ is called \emph{injective} if every topological linear embedding $F \rightarrow G$, $G$ a Banach space, admits a continuous linear left inverse, while $F$ is called  \emph{projective} if every surjective continuous linear map $G \rightarrow F$, $G$ a Banach space, admits a continuous linear right inverse. 

	\begin{lemma}\label{l:StepwiseEAcyclicBanachSpace}

		Let $E$ be a lcHs and  let $F$  be a Banach space. Suppose that $E$ is an injective Banach space or that $F$ is a projective Banach space. Then, $F$ is weakly $E$-acyclic in the category of Banach spaces.
	\end{lemma}
	\begin{proof}
		If $E$ is an injective Banach space, then every short exact sequence
			\[ \SES{E}{G}{F}{}{} \]
		of Banach spaces splits. The same push-out argument as in the proof of Proposition \ref{p:Splitting=EAcyclic} below shows that $F$ is weakly $E$-acyclic in the category of Banach spaces (see also  \cite[Proposition 5.1.3]{W-DerivFunctFuncAnal}). 
		Next, suppose that $F$ is a projective Banach space.  Let 
			\[ \SES{G}{H}{F}{\iota}{} \]
		be an exact sequence of Banach spaces. Since $F$ is projective, there exists $P \in L(H, G)$ such that $P \circ \iota = \id_{G}$. Then, for every $T \in L(G, E)$ it holds that $S= T \circ P \in L(H, E)$ satisfies $T = S \circ \iota$, whence $\iota_*: L(H,E) \to L(G,E)$ is surjective.
	\end{proof}
	
	
%
	
	\begin{lemma}
		\label{c:InjectiveLimitInjOrProjBanachSpaces=locallyEacyclic}
		Let $E$ be a lcHs and let $Z$ be an $(LB)$-space. Suppose that  $E$ is an $(LB)$-space such that $E = \Ind \mathscr{E}$ for some inductive spectrum $\mathscr{E} = (E_n)_{n \in \N}$  of Banach spaces such that all inclusion maps $E_n \to E_{n+1}$ factor through an injective Banach space or that $Z = \Ind \mathscr{Z}$ for some inductive spectrum $\mathscr{Z} = (Z_n)_{n \in \N}$  of Banach spaces such that all inclusion maps $ Z_n \to Z_{n+1}$ factor through a projective Banach space.  Then, $Z$ is locally weakly $E$-acyclic.
	\end{lemma}
	\begin{proof}
In the second case, the result is a direct consequence of Lemma \ref{l:StepwiseEAcyclicBanachSpace}. In the first case, we note that Grothendieck's factorization theorem and Lemma \ref{l:StepwiseEAcyclicBanachSpace}  imply that any Banach space is weakly $E$-acyclic in the category of Banach spaces, in particular $Z$ is locally weakly $E$-acyclic.
	\end{proof}

	\begin{corollary}
		\label{c:SpecificCasesAreLocallyEAcyclic}
		Let $E$ be a lcHs and let $Z$ be an $(LB)$-space. Suppose that one of the following assumptions is satisfied:
		\begin{itemize} 
			\item[(a)] $E \cong k^{\infty}(A)$.
			\item[(b)]   $E$  is an $(LN)$-space.
			\item[(c)]  $Z \cong k^{1}(A)$.
			\item[(d)]  $Z$ is an $(LN)$-space.
			\end{itemize}		
		Then, $Z$ is locally $E$-acyclic. 
	\end{corollary}
	
	\begin{proof} We have that $\ell^{\infty}$ is  an injective Banach space \cite[p.~105]{L-T-ClassicalBanachSpI} and that $\ell^{1}$ is a projective Banach space \cite[Proposition 2.f.7, p.~107]{L-T-ClassicalBanachSpI}.
Hence, the cases (a) and (c) are therefore a direct consequence of Lemma \ref{c:InjectiveLimitInjOrProjBanachSpaces=locallyEacyclic}. Since any nuclear map between two Banach spaces can be factored through $\ell^p(\N)$, $p=1, \infty$, (b) and (d) also follow from this lemma. 
	\end{proof}
	
As a first application of Theorem \ref{t:MittagLeffler} and Corollary \ref{c:SpecificCasesAreLocallyEAcyclic}, we characterize (weak) acyclicity of an $(LB)$-space in terms of weak $\ell^{\infty}(I)$-acyclicity.

\begin{proposition}
		\label{t:MainExtensionThm2}
		Let $Z$ be an $(LB)$-space. Then, 
			\begin{itemize} 
			\item[(i)] $Z$ is acyclic if and only if $Z$ is weakly $\ell^{\infty}(I)$-acyclic for every index set  $I$. 
			\item[(ii)]   $Z$ is weakly acyclic if and only if $Z$ is weakly $\ell^{\infty}(I)$-acyclic for every (some) finite index set  $I$.
			\end{itemize}
\end{proposition}
\begin{proof}
It holds that $Z$ is acylic if and only if $\Proj^{1} \mathscr{L}(Z, \ell^{\infty}(I)) = 0$ for every index set  $I$, while  $Z$ is weakly acyclic if and only if $\Proj^{1} \mathscr{L}(Z, \C) = 0$  \cite[p.\ 110]{W-DerivFunctFuncAnal}.	Moreover, it is clear that $Z$ is weakly $\ell^{\infty}(I)$-acyclic for every finite index set  $I$ if and only if $Z$ is weakly $\C$-acyclic.		
Hence, the result follows from Theorem \ref{t:MittagLeffler} and Corollary \ref{c:SpecificCasesAreLocallyEAcyclic}. \end{proof}
We end this subsection by giving two corollaries of Proposition \ref{t:MainExtensionThm2} that will be used later on.
\begin{corollary}
		\label{c:wa}
		Let $E$ be a lcHs and let $Z$ be a weakly $E$-acyclic $(LB)$-space. Then, $Z$ is weakly acyclic.
	\end{corollary}
\begin{proof}
As $Z$ is weakly $E$-acyclic, $Z$ is weakly $E_0$-acyclic for every lcHs $E_0$ that is isomorphic to a complemented subspace of $E$. In particular, $Z$ is weakly $\C$-acyclic and thus, by Proposition \ref{t:MainExtensionThm2},  weakly acyclic.
\end{proof}

 The next result is essentially due to Palamodov \cite[\S 6]{P-HomMethTheoryLCS} (see also \cite[Proposition 1.2]{V-RegPropLFSp}). We present a short proof based on  Proposition \ref{t:MainExtensionThm2}.
 \begin{corollary}
		\label{p:AcyclicEquiv}
		An $(LB)$-space $Z$ is (weakly) acyclic if and only if for every exact sequence 
			\begin{equation}
			\label{exactLB}
			\SES{X}{Y}{Z}{\iota}{} \end{equation}
		of $(LB)$-spaces the map $\iota$ is a (weak) topological embedding.
		\end{corollary}
	
	\begin{proof}
		By considering the injective resolution \eqref{eq:InjectiveResolution} of an inductive spectrum of Banach spaces defining $Z$, we find that the condition is sufficient. If $Z$ is weakly acyclic, the converse follows directly from Proposition \ref{t:MainExtensionThm2}(ii). Suppose now that $Z$ is acyclic and consider an exact sequence \eqref{exactLB} of $(LB)$-spaces. Let $\{x'_{j}\}_{j \in I}$ be an arbitrary equicontinuous subset of $X'$ and define $T \in L(X,\ell^\infty(I))$  via $T(x) =  (\ev{x'_{j}}{x})_{j \in I}$, $x \in X$.  Proposition \ref{t:MainExtensionThm2}(i)  yields that $Z$ is weakly $\ell^{\infty}(I)$-acylic. Hence, there is $S \in L(Y, \ell^{\infty}(I))$ such that $S \circ \iota = T$. Note  that there is an equicontinuous subset $\{y'_{j}\}_{j \in I}$ of $Y^{\prime}$ such that $S(y) = (\ev{y'_{j}}{y})_{j \in I}$, $y \in Y$. Thus $\iota^t(y'_{j}) = x'_{j}$ for all $j \in I$. Consequently, every equicontinuous subset of $X^{\prime}$ is the image under $\iota^t$ of an equicontinuous subset of $Y^{\prime}$. Let  $U$ be an arbitrary absolutely convex closed neighborhood of zero in $X$. Then, $U^{\circ} \subseteq X^{\prime}$ is equicontinuous,  whence there is some equicontinuous subset $A$ of $Y^{\prime}$ such that $\iota^{t}(A) = U^\circ$.  Consider the zero neighborhood $V= A^\circ$  of $Y$. The bipolar theorem implies that $V \cap \iota(X) = \iota(U)$. This shows that $\iota$ is open.
	\end{proof}

\subsection{A characterization of $\Proj^{1} \mathscr{L}(Z, E) = 0$}\label{subsect-2}

Throughout this subsection, we fix a locally complete lcHs $E$ with a fundamental sequence of bounded sets and an $(LB)$-space $Z$. Let $(B_{N})_{N \in \N}$ be a fundamental increasing sequence of bounded sets in $E$ such that each $B_N$ is a Banach disk and set $E_{N} = E_{B_N}$. Let $\mathscr{Z} = (Z_{n})_{n \in \N}$ be an inductive spectrum of Banach spaces such that $Z = \Ind \mathscr{Z}$. 

For $n,N \in \N$ we endow $L(Z_n,E_N)$ with its natural Banach space topology and denote by $B(Z_{n}, E_{N})$ the unit ball in $L(Z_n,E_N)$. 
Then, $(L(Z_n,E_N))_{N \in \N}$ is an inductive spectrum of Banach spaces and 
$$L(Z_n, E) = \bigcup_{N \in \N} L(Z_n, E_{N}).$$ 
We endow $L(Z_n, E)$ with the corresponding inductive limit topology. Then,  $L(Z_n, E)$ is an $(LB)$-space (it is Hausdorff as $E$ is so). In such a way, the projective spectrum $\mathscr{L}(\mathscr{Z}, E) = (L(Z_{n}, E), \varrho^n_m)$ consists of $(LB)$-spaces and continuous linear maps $\varrho^n_m$.  Proposition \ref{P:CondProj1=0} applied to the projective spectrum $\mathscr{L}(\mathscr{Z}, E)$ reads as follows.

	
%
%
	
	\begin{proposition}
		\label{l:SuffNeccCondProj1=0}
		Consider the conditions
			\begin{equation}
				\label{eq:SuffCondProj1=0}
				\begin{gathered}
					\forall n \in \N ~ \exists m \geq n ~ \forall k \geq m ~ \exists N \in \N ~ \forall M \geq N, \varepsilon > 0 ~ \exists K \geq M, C > 0  \\
					\forall T \in B(Z_{m}, E_{M})  ~ \exists S \in CB(Z_{k}, E_{K}), R \in \varepsilon B(Z_{n}, E_{N}) \, : \,
					\mbox{$T =S + R$ on $Z_n$.}
				\end{gathered}
			\end{equation}
		and
			\begin{equation}
				\label{eq:NeccCondProj1=0}
				\begin{gathered}
					\forall n \in \N ~ \exists m \geq n ~ \forall k \geq m ~ \exists N \in \N ~ \forall M \geq N ~ \exists K \geq M, C > 0 \\
\forall T \in B(Z_{m}, E_{M})  ~ \exists S \in CB(Z_{k}, E_{K}), R \in CB(Z_{n}, E_{N}) \, : \,
					\mbox{$T =S + R$ on $Z_n$.}				\end{gathered}
			\end{equation}
		Then, \eqref{eq:SuffCondProj1=0} $\Rightarrow$ $\Proj^{1} \mathscr{L}(Z, E) = 0$ $\Rightarrow$ \eqref{eq:NeccCondProj1=0}.
	\end{proposition}
	

In this subsection, we will use Proposition \ref{l:SuffNeccCondProj1=0}  to relate  $\Proj^{1} \mathscr{L}(Z, E) = 0$ to the conditions $(\condS)$ and $(\condWS)$  for $(E, Z)$. We start with the following result.

	\begin{lemma}
		\label{l:Proj1=0ImpliesS}
		If $\Proj^{1} \mathscr{L}(Z, E) = 0$, then $(E, Z)$ satisfies $(\condWS)$.
	\end{lemma}

	\begin{proof}
		By Proposition  \ref{l:SuffNeccCondProj1=0}, it suffices to show that \eqref{eq:NeccCondProj1=0} implies that  $(E, Z)$ satisfies $(\condWS)$. To this end, it is enough to prove that, for $n \leq m \leq k$, $N \leq M \leq K$, $C > 0$ fixed, it holds that  the condition
		$$\forall T \in B(Z_{m}, E_{M})  ~ \exists S \in CB(Z_{k}, E_{K}), R \in CB(Z_{n}, E_{N}) \, : \,\mbox{$T =S + R$ on $Z_n$}
		$$
		implies that 	
	$$ \|z\|_{Z_{m}} B_{M} \subseteq C  (\|z\|_{Z_{k}} B_{K} + \|z\|_{Z_{n}} B_{N}) , \qquad \forall z \in Z_{n} .
	$$		
	Let  $z \in Z_{n} \setminus \{0\}$ and  $e \in B_{M}$ be arbitrary. By the Hahn-Banach theorem, there is $z' \in Z^{\prime}_{m}$ with $\|z'\|_{Z^{\prime}_{m}} \leq 1$ such that $\ev{z'}{z} = \|z\|_{Z_{m}}$. Define  $T \in B(Z_{m}, E_{M})$ via $T(v) = \ev{z'}{v} e$, $v \in Z_m$. Hence, there are $S \in C B(Z_{k}, E_{K})$ and $R \in C B(Z_{n}, E_{N})$ such that $T =S + R$ on $Z_n$. Then,
	$$\|z\|_{Z_{m}} e = T(z) = S(z) + R(z) \in  C  (\|z\|_{Z_{k}} B_{K} + \|z\|_{Z_{n}} B_{N}).$$
	\end{proof}
	
Our next goal is to show that  $(\condS)$ for $(E, Z)$ implies  that $\Proj^{1} \mathscr{L}(Z, E) = 0$ if one of the assumptions (a)-(d) from Corollary  \ref{c:SpecificCasesAreLocallyEAcyclic} is satisfied.
We will often use the following technical lemma.

	\begin{lemma}
		\label{l:CondSImplications}
		Let $F_{0} \subseteq F_{1} \subseteq F_{2}$ be three seminormed spaces with continuous inclusion maps.  
		Set $U_{j} = \{ x \in F_{j} \mid \|x\|_{F_{j}} \leq 1 \}$, $j = 0, 1, 2$. Let $C, \varepsilon > 0$. If
		$$\|x\|_{F_{1}} \leq C \|x\|_{F_{2}} + \varepsilon \|x\|_{F_{0}}, \qquad \forall x \in F_{0},$$
		then
			\[ \forall x'_{1} \in U_{1}^{\circ} ~ \exists x'_{2} \in 2CU_{2}^{\circ}, x'_{0} \in 2\varepsilon U_{0}^{\circ} \, : \,\mbox{$x'_1 =x'_2 + x'_0$ on $F_0$}. \]
	\end{lemma}
	
	\begin{proof}
	Set $V_j = \{x \in F_{0} \mid \|x\|_{F_{j}} \leq 1 \}$ for $j =1, 2$.  
	Let $V_j^\circ$ be the polar set of $V_j$ in $F'_0$. By the Hahn-Banach theorem, it suffices to show that $V_1^\circ \subseteq 2CV_2^\circ + 2\varepsilon U^\circ_0$. Our assumption implies that $\frac{1}{2C}V_2 \cap \frac{1}{2\varepsilon}U_0  \subseteq V_1$.  Hence, the result follows from the bipolar theorem and the  Banach-Alaoglu theorem.
			\end{proof}
	
	\begin{lemma}\label{l:Siss}
		Suppose that  $(E, Z)$ satisfies $(\condS)$  and that one of the assumptions (a)--(d) from Corollary  \ref{c:SpecificCasesAreLocallyEAcyclic} holds. Then, $\Proj^{1} \mathscr{L}(Z, E) = 0$.
	\end{lemma}
	
	\begin{proof}
		By Lemma \ref{l:SuffNeccCondProj1=0} it suffices to show that $(\condS)$ for  $(E, Z)$  implies \eqref{eq:SuffCondProj1=0}. We consider each of the cases (a)--(d) separately.
		
	\noindent (a) Let  $A = (a_{N})_{N \in \N} = (a_{N, j})_{N \in \N, j \in I}$.  We may suppose that $B_N = ND_N$ with $D_N$ the closed unit ball in $\ell^\infty(a_N)$ and thus that $E_N =\ell^\infty(a_N)$ endowed with the norm $\frac{1}{N}\| \, \cdot \, \|_{\ell^\infty(a_N)}$. It suffices to show that, for $n \leq m \leq k$, $N \leq M \leq K$, $C, \varepsilon > 0$ fixed,
			\begin{equation}
			\label{Sinfty} 
			\|z\|_{Z_{m}} B_{M} \subseteq C \|z\|_{Z_{k}} B_{K} + \varepsilon \|z\|_{Z_{n}} B_{N} , \qquad \forall z \in Z_{n},  \end{equation}
			implies that 
			$$
			\forall T \in B(Z_{m},E_M)  ~ \exists S \in 2CB(Z_{k}, E_K), R \in 2\varepsilon B(Z_{n}, E_N) \, : \,
					\mbox{$T =S + R$ on $Z_n$.}
			$$		
		Condition \eqref{Sinfty} implies that
			\[\frac{M}{a_{M, j}} \|z\|_{Z_{m}} \leq C \frac{K}{a_{K, j}} \|z\|_{Z_{k}} + \varepsilon \frac{N}{a_{N, j}} \|z\|_{Z_{n}} , \qquad \forall z \in Z_{n} . \]
		Let $T \in B(Z_{m}, E_{M})$ be arbitrary. Then, there are $z'_{j} \in Z^{'}_{m}$, $j \in I$, such that $T(z) = (\ev{z'_{j}}{z})_{j \in I}$, $z \in Z_m$, and $\|z'_{j}\|_{Z^{\prime}_{m}} \leq M/ a_{M, j}$ for all $j \in I$. By Lemma \ref{l:CondSImplications}, we find $z'_{j, 2} \in Z^{\prime}_{k}$ with $\|z'_{j, 2}\|_{Z^{\prime}_{k}} \leq 2C K / a_{K, j}$ and $z'_{j, 0} \in Z^{\prime}_{n}$ with $\|z'_{j, 0}\|_{Z^{\prime}_{n}} \leq 2 \varepsilon N / a_{N, j}$ such that $z'_{j} = z'_{j, 2} + z'_{j, 0}$ on $Z_{n}$. Define $S(z) = (\ev{z'_{j, 2}}{z})_{j \in I}$, $z \in Z_k$, and $R(z) = (\ev{z'_{j, 0}}{z})_{j \in I}$, $z \in Z_N$.  Then, $S \in 2C B(Z_{k}, E_{K})$ and $R \in 2 \varepsilon B(Z_{n}, E_{N})$ are such that $T =S + R$ on $Z_n$.
\\ \\
\noindent (b) We may suppose that each $E_{N}$ is a Hilbert space (with inner product $\langle \cdot, \cdot \rangle_N$) and that the inclusion maps $i_N: E_{N} \rightarrow E_{N+1}$ are nuclear for all $N \in \N$.
We distinguish two cases. \\
\noindent \emph{Case 1: The inductive spectrum $(E_{N})_{N \in \N}$ satisfies
	\begin{equation} 
				\label{extra-cond}
				\exists N \in \N \, \forall M \geq N \, \exists K \geq M\,:\, E_M \subseteq \overline{E_N}^{E_K}.
\end{equation}}
Condition $(\condS)$ for $(E, Z)$ and \eqref{extra-cond}  together yield that
\begin{gather*}
							\forall n \in \N ~ \exists m \geq n ~ \forall k \geq m ~ \exists N \in \N ~ \forall M \geq N, \varepsilon >0~ \exists K \geq M, C > 0 \,: \\
							 \forall z \in Z_n \, : \, \|z\|_{Z_{m}} B_{M} \subseteq C \|z\|_{Z_{k}} B_{K} + \varepsilon \|z\|_{Z_{n}} B_{N} \quad \mbox{and} \qquad E_M \subseteq \overline{E_N}^{E_K}.
						\end{gather*}
Hence, it suffices to show that, for $n \leq m \leq k$, $N + 2 \leq M < K$ fixed, there are $C_0, C_1 > 0$ such that for all $C, \varepsilon > 0$ 
	\begin{equation} 
				\label{eq:SImpliesProj1=0--E(DFN)--Cond1}
				 \forall z \in Z_n \, : \, \|z\|_{Z_{m}} B_{M} \subseteq C \|z\|_{Z_{k}} B_{K} + \varepsilon \|z\|_{Z_{n}} B_{N} \quad \mbox{and} \quad E_M \subseteq \overline{E_{N+1}}^{E_{K+1}}
			\end{equation}
			imply that
			$$
			\forall T \in B(Z_{m},E_M)  ~ \exists S \in CC_0B(Z_{k}, E_{K+2}), R \in  \varepsilon C_1 B(Z_{n}, E_{N+2}) \, : \,
					\mbox{$T =S + R$ on $Z_n$.}
			$$		
From the first property in  \eqref{eq:SImpliesProj1=0--E(DFN)--Cond1} we obtain that
			\begin{equation} 
				\label{eq:SImpliesProj1=0--E(DFN)--Cond2}
				\|z\|_{Z_{m}} \|e'_{|E_M}\|_{E^{\prime}_{M}} \leq C \|z\|_{Z_{k}} \|e'\|_{E^{\prime}_{K}} + \varepsilon \|z\|_{Z_{n}} \|e'_{|E_N}\|_{E^{\prime}_{N}}, \qquad \forall z \in Z_{n}, e' \in E^{\prime}_{K} .  
			\end{equation}
By considering a singular value decomposition of the compact inclusion map $E_{N+1} \rightarrow E_{K+1}$, we find a complete orthonormal system $(e_{j})_{j \in I}$ in $E_{N+1}$ that is orthogonal in $E_{K+1}$.  
Set  $f_j = e_j/\| e_j\|_{E_{K+1}}$, $j \in I$. Note that 
\begin{equation}
\label{swichtinpr}
\langle e, f_j \rangle_{K+1} =  \sum_{j' \in I} \ev{e}{e_{j'}}_{N + 1} \langle e_{j'}, f_j \rangle_{K+1}  = \|e_{j}\|_{E_{K+1}} \ev{e}{e_{j}}_{N + 1}, \quad \forall e \in E_{N+1}.
\end{equation}
Hence, 
$$
e = \sum_{j \in I} \langle e, f_j \rangle_{K+1} f_j, \qquad \forall e \in E_{N+1},
$$
where the series converges in $E_{N+1}$.
This equality together with the second property in  \eqref{eq:SImpliesProj1=0--E(DFN)--Cond1} and the fact that $(f_{j})_{j \in I}$ is an orthonormal system in $E_{K+1}$ imply that 
\begin{equation}
\label{equality-lift}
e = \sum_{j \in I} \langle e, f_j \rangle_{K+1} f_j, \qquad \forall e \in E_{M},
\end{equation}
where the series converges in $E_{K+1}$.
Let $T \in B(Z_{m}, E_{M})$ be arbitrary. For $j \in I$ we define $f'_j \in E'_{K+1}$ via $\langle f'_j, e \rangle = \langle e, f_j \rangle_{K+1}$, $e \in E_{K+1}$, and set $z'_j = f'_j \circ  T \in Z'_m$. Note that $\|z'_j\|_{Z'_m} \leq  \|f'_{j|E_M}\|_{E^{\prime}_{M}}$. By \eqref{eq:SImpliesProj1=0--E(DFN)--Cond2} and Lemma \ref{l:CondSImplications}, we find $v'_j \in Z'_k$ with $\|v'_j\|_{Z'_k} \leq 2C\|f'_{j|E_K}\|_{E^{\prime}_{K}}$ and $w'_j \in Z'_n$ with $\|w'_j\|_{Z'_n} \leq 2\varepsilon\|f'_{j|E_N}\|_{E^{\prime}_{N}}$ such that $z'_j = v'_j +w'_j$ on $Z_n$.
 Define 
\begin{gather*}
S(z) = \sum_{j \in I} \langle v'_j, z \rangle f_{j}, \qquad z \in Z_k, \\
 R(z) = \sum_{j \in I}  \langle w'_j, z \rangle f_{j}, \qquad  z \in Z_n.
 \end{gather*}
		In order to show the required continuity estimates for $S$ and $R$, we need some preparation. The inclusion maps $i_J: E_{J} \rightarrow E_{J+1}$ are nuclear and thus Hilbert-Schmidt. Let $L \in \N$ be arbitrary and let $g =(g_{j})_{j \in I}$ be any orthonormal system in $E_{L + 1}$. Then,
					\[ \sum_{j \in I} |\ev{e}{g_{j}}_{L + 1}| \|g_{j}\|_{E_{L + 2}} \leq \nu_{2}(i_{L + 1}) \|e\|_{E_{L + 1}},\qquad \forall e \in E_{L+1}, \]
		where $\nu_{2}$ denotes the Hilbert-Schmidt norm. Hence, 
			\[ \sigma_{L + 1} : E_{L + 1} \rightarrow \ell^{1}, \, e \mapsto (\ev{e}{g_{j}}_{L + 1} \| g_{j}\|_{E_{L + 2}})_{j \in I} \]
		is a well-defined continuous linear map with $ \|\sigma_{L + 1}\|_{L(E_{L + 1}, \ell^{1})} \leq \nu_{2}(i_{L + 1})$. Since the map $\sigma_{L + 1} \circ i_{L}: E_{L} \rightarrow \ell^{1}$ is nuclear  (as  $i_{L}$ is nuclear), there exists a positive sequence $(\lambda^{L,g}_{j})_{j \in I} \in \ell^1$ such that
			\[ |\ev{e}{g_{j}}_{L + 1} |\|g_{j}\|_{E_{L + 2}} \leq \lambda^{L,g}_{j} \|e\|_{E_{L}} , \qquad \forall e \in E_{L}, j \in I, \]
		and
			\[ \sum_{j \in I} \lambda^{L,g}_{j} \leq 2 \nu_{1}(\sigma_{L + 1} \circ i_{L}) \leq 2 \|\sigma_{L + 1}\|_{L(E_{L + 1}, \ell^{1})} \nu_{1}(i_{L}) \leq 2 \nu_{1}(i_{L}) \nu_{2}(i_{L + 1}) = C_L , \]
		 where $\nu_{1}$ denotes the nuclear norm. Set $\gamma^{L,g}_{j} = \| g_{j}\|_{E_{L + 2}} / \lambda^{L,g}_{j}$. Then, 
			\begin{equation}
				\label{eq:GammaUpperBound} 
				\sup_{j \in I} \gamma^{L,g}_{j} |\ev{e}{g_{j}}_{L + 1}| \leq \|e\|_{E_{L}}, \qquad \forall e \in E_L. 
			\end{equation}
		 For $L = K$ and $(g_{j})_{j \in I} = (f_{j})_{j \in I}$, \eqref{eq:GammaUpperBound} gives that
			$$
				\sup_{j \in I} \gamma^{(K)}_{j} \|f'_{j|E_K}\|_{E^{\prime}_{K}} \leq 1,
			$$
			where $\gamma^{(K)}_{j} = \gamma^{K,f}_{j}$ with $f =  (f_{j})_{j \in I}$. 
	In view of \eqref{swichtinpr}, \eqref{eq:GammaUpperBound} with $L = N$ and $(g_{j})_{j \in I} = (e_{j})_{j \in I}$ implies that
		$$
				\sup_{j \in I}  \frac{\gamma^{(N)}_{j}}{\|e_{j}\|_{E_{K+1}} } \|f'_{j|E_N}\|_{E^{\prime}_{N}} \leq 1,
		$$
		where $\gamma^{(N)}_{j} = \gamma^{N,e}_{j}$ with $e=  (e_{j})_{j \in I}$. 
			We return to $S$ and $R$. We have that for all $z \in Z_k$
			\begin{align*}
				\|S(z)\|_{E_{K + 2}}
				&\leq \sum_{j \in I} |\langle v'_j, z \rangle| \| f_{j}\|_{E_{K + 2}}
				\leq C_{K} \sup_{j \in I} \gamma^{(K)}_{j}  |\langle v'_j, z \rangle| \\
				&\leq 2CC_{K}   \|z\|_{Z_{k}} \sup_{j \in I} \gamma^{(K)}_{j}  \|f'_{j|E_K}\|_{E'_K} 
				\leq   2CC_{K} C \|z\|_{Z_{k}} ,
			\end{align*}
		whence $S \in  2CC_K B(Z_{k}, E_{K + 2})$. Similarly, we find for all $z \in Z_n$
			\begin{align*}
				\|R(z)\|_{E_{N + 2}}
				&\leq \sum_{j \in I} \frac{1}{\|e_{j}\|_{E_{K+1}} } |\langle w'_j, z \rangle| \| e_{j}\|_{E_{N + 2}}
				\leq C_{N} \sup_{j \in I} \gamma^{(N)}_{j} \frac{1}{\|e_{j}\|_{E_{K+1}} } |\langle w'_j, z \rangle| \\
				&\leq 2 \varepsilon C_N\|z\|_{Z_{n}} \sup_{j \in I} \frac{ \gamma^{(N)}_{j}}{\|e_{j}\|_{E_{K+1}} } \|f'_{j|E_N}\|_{E'_N} 
				\leq 2 \varepsilon C_N \|z\|_{Z_{n}} ,
			\end{align*}
		and thus $R \in  2\varepsilon C_NB(Z_{n}, E_{N+ 2})$.  Finally, by \eqref{equality-lift}, we have that for all $z \in Z_n$
	\begin{align*} 
				S(z) + R(z) &= \sum_{j \in I} \langle z'_j, z \rangle f_{j}  = \sum_{j \in I} \langle f'_j, T(z) \rangle f_{j}  =  \sum_{j \in I} \langle  T(z),f_j \rangle_{K+1} f_{j} =  T(z).
			\end{align*}			
\noindent \emph{Case 2: The inductive spectrum $(E_{N})_{N \in \N}$ does not satisfy \eqref{extra-cond}, i.e.,
	\begin{equation} 
				\label{extra-cond3}
				\forall N \in \N \, \exists M \geq N \, \forall K \geq M~:~ E_M \not \subseteq \overline{E_N}^{E_K}.
\end{equation}}
Condition $(\condS)$ for $(E, Z)$ implies that
\begin{gather*}
							\forall n \in \N ~ \exists m \geq n ~ \forall k \geq m ~ \exists N \in \N ~ \forall M \geq N, \varepsilon >0~ \exists K \geq M, C > 0 \,: \\
				\|z\|_{Z_{m}} \|e'_{|E_M}\|_{E^{\prime}_{M}} \leq C \|z\|_{Z_{k}} \|e'\|_{E^{\prime}_{K}} + \varepsilon \|z\|_{Z_{n}} \|e'_{|E_N}\|_{E^{\prime}_{N}}, \qquad \forall z \in Z_{n}, e' \in E^{\prime}_{K} .  							
\end{gather*}
The Hahn-Banach theorem and \eqref{extra-cond3} yield that
$$
\forall N \in \N \, \exists M \geq N \, \forall K \geq M \, \exists e' \in E'_K\,:\,e'_{\mid E_N} = 0 \quad \mbox{and} \quad e'_{\mid E_M} \neq 0.
$$
By combining the above two properties, we obtain that
\begin{equation}
\label{strict?}
\forall n \in \N \, \exists m \geq n \, \forall k \geq m \, \exists C > 0 \, \forall z \in Z_n\,:\,\|z\|_{Z_m} \leq C\|z\|_{Z_k}.
\end{equation}
We will show that 
\begin{gather*}
\forall n \in \N \, \exists m \geq n \, \forall k \geq m \, \forall M \in \N \, \exists C >0  \\
 \forall T \in B(Z_{m},E_M)  ~ \exists S \in CB(Z_{k}, E_{M+1})  \, : \,	\mbox{$T =S$ on $Z_n$,}
\end{gather*}
which implies \eqref{eq:SuffCondProj1=0}. Let $n \in \N$ be arbitrary and choose $m \geq n$ according to \eqref{strict?}. Let $k \geq m$ and $M \in \N$ be arbitrary.  Since the inclusion map $E_M \to E_{M+1}$ is nuclear, there are sequences $(e_j)_{j \in \N} \subseteq E_M$ and $(f_j)_{j \in \N} \subseteq E_{M+1}$ such that
$$
C_1 = \sup_{j \in \N} \|e_j\|_{E_M} < \infty, \qquad C_2 = \sum_{j =0}^\infty \|f_j\|_{E_{M+1}} < \infty,
$$
and
$$
e = \sum_{j =0}^\infty \langle e, e_j \rangle_{E_M} f_j, \qquad e \in E_M.
$$
Let $T \in B(Z_{m},E_M)$ be arbitrary. For $j \in \N$ we define $z'_j \in Z'_m$ via $\langle z'_j,z\rangle = \langle T(z), e_j \rangle_{M}$, $z \in Z_m$. Note that $\|z'_j \|_{Z'_m} \leq C_1$ for all $j \in\N$. The Hahn-Banach theorem and \eqref{strict?} imply that
that there exists $v'_j \in Z'_k$ such that $v'_{j \mid Z_n} = z'_j$ and $\|v'_j \|_{Z'_k} \leq CC_1$ for all $j \in \N$. Define
$$
S(z) = \sum_{j =0}^\infty \langle v'_j, z \rangle f_{j}, \qquad z \in Z_k.
$$
Then,
$$
\|S(z)\|_{E_{M+1}} \leq \sum_{j = 0}^{\infty} |\langle v'_j, z \rangle| \| f_{j}\|_{E_{M+1}} \leq CC_1C_{2} \|z\|_{Z_k}, \qquad z \in Z_k,
$$
		whence $S \in  CC_1C_2 B(Z_{k}, E_{M+1})$. Furthermore, for all $z \in Z_n$ 
		$$
	S(z) = \sum_{j =0}^\infty \langle v'_j, z \rangle f_{j} = \sum_{j =0}^\infty \langle T(z),e_j \rangle_{M} f_j = T(z). 	
		$$

\noindent (c) Let $A = (a_{n})_{n \in \N} = (a_{n, j})_{n \in \N, j \in I}$. We may suppose that $Z_n  = \ell^1(a_n)$ for $n \in \N$. It suffices to show that, for $n \leq m \leq k$, $N \leq M \leq K$, $C, \varepsilon > 0$ fixed,
			\begin{equation}
			\label{S1} 
			\|c\|_{\ell^1(a_m)} B_{M} \subseteq C \|c\|_{\ell^1(a_k)} B_{K} + \varepsilon \|c\|_{\ell^1(a_n)} B_{N} , \qquad \forall c \in \ell^1(a_n),  \end{equation}
			implies that 
			$$
			\forall T \in B(\ell^1(a_m),E_M)  ~ \exists S \in CB(\ell^1(a_k), E_K), R \in \varepsilon B(\ell^1(a_n), E_N) \, : \, \mbox{$T =S + R$ on $\ell^1(a_n)$.}
			$$		
			For $j \in I$ we define $c^{(j)} = (\delta_{j, j'})_{j' \in I}$. By evaluating \eqref{S1} at $c= c^{(j)}$, we obtain that
			\[ a_{m, j} B_{M} \subseteq C a_{k, j} B_{K} + \varepsilon a_{n, j} B_{N} , \qquad \forall j \in I .  \]
		Let $T \in B(\ell^{1}(a_{m}), E_{M})$ be arbitrary. Define $e_j = T(c^{(j)})$ and note  that $T(c) = \sum_{j \in I} c_{j} e_{j}$ for all $c = (c_{j})_{j \in I} \in \ell^{1}(a_{m})$. Since $e_{j} \in a_{m, j} B_{M}$, we find $f_{j} \in C a_{k, j} B_{K}$ and $g_{j} \in \varepsilon a_{n, j} B_{N}$ such that $e_{j} = f_{j} + g_{j}$. Define $S(c) = \sum_{j \in I} c_{j} f_{j}$,  $c = (c_{j})_{j \in I} \in \ell^{1}(a_{k})$, and $R(c) = \sum_{j \in I} c_{j} g_{j}$,  $c = (c_{j})_{j \in I} \in \ell^{1}(a_{n})$. Then, $S \in CB(\ell^1(a_k), E_K)$ and $R \in \varepsilon B(\ell^1(a_n), E_N)$ are such that $T =S + R$ on $\ell^1(a_n)$.
\\ \\	
\noindent (d) We may suppose that each $Z_{n}$ is a Hilbert space (with inner product $\langle \cdot, \cdot \rangle_n$) and that the inclusion maps $i_n : Z_{n} \rightarrow Z_{n+1}$ are Hilbert-Schmidt.
 It suffices to show that, for $n < m < k$, $N \leq M \leq K$, $C, \varepsilon > 0$ fixed,
			$$
			\|z\|_{Z_{m}} B_{M} \subseteq C \|z\|_{Z_{k+1}} B_{K} + \varepsilon \|z\|_{Z_{n+1}} B_{N} , \qquad \forall z \in Z_{n}, $$
			implies that 
			$$
			\forall T \in B(Z_{m},E_M)  ~ \exists S \in CC_0B(Z_{k}, E_K), R \in \varepsilon C_1 B(Z_{n}, E_N) \, : \,
					\mbox{$T =S + R$ on $Z_n$,}
			$$		
where $C_0 =\nu_{2}(i_{k})  $ and $C_1 = \nu_{2}(i_{n})$.
By considering a singular value decomposition of the compact inclusion map $Z_{n} \rightarrow Z_{k}$, we find a complete orthonormal system $(z_{j})_{j \in I}$ in $Z_{n}$ that is orthogonal in $Z_{k}$. 
Put $w_{j} = z_{j} / \|z_j\|_{Z_{k}}$ for $j \in I$. Then $(w_{j})_{j \in I}$ is an orthonormal system in $Z_{k}$ and note that
\begin{equation}
\label{eq12}
 \ev{z}{w_{j}}_{k} = \sum_{j' \in I} \ev{z}{z_{j'}}_{n} \ev{z_{j'}}{w_{j}}_{k} = \ev{z}{z_{j}}_{n} \ev{z_{j}}{w_{j}}_{k} = \| z_j\|_{Z_k} \ev{z}{z_{j}}_{n}, \quad  \forall z \in Z_{n}.
 \end{equation}
Let $T \in B(Z_{m}, E_{M})$ be arbitrary. Set  $e_{j} = T(z_j) \in \| z_j\|_{Z_m} B_{M}$. There are $f_{j} \in C \|z_j\|_{Z_{k+1}} B_{K} $ and $g_j \in \varepsilon \|z_j\|_{Z_{n+1}} B_{N}$ such that $e_{j} = f_{j} + g_{j}$. Define 
\begin{gather*}
S(z) = \sum_{j \in I} \frac{1}{\| z_j\|_{Z_k}} \langle z,w_j \rangle_k f_{j}, \qquad z \in Z_k, \\
 R(z) = \sum_{j \in I}  \langle z,z_j \rangle_n g_{j}, \qquad  z \in Z_n.
 \end{gather*}
  For all $z \in Z_k$ it holds that
			\begin{align*} 
				\|S(z)\|_{E_{K}} 
				&\leq \sum_{j \in I} |\ev{z}{w_{j}}_{k}| \left(  \frac{1}{\| z_j\|_{Z_k}} \|f_{j}\|_{E_{K}} \right)  \\
				&\leq C \left( \sum_{j \in I} |\ev{z}{w_{j}}_{k}|^{2} \right)^{1/2} \left( \sum_{j \in I} \|w_{j}\|^{2}_{Z_{k + 1}} \right)^{1/2} \\
				&\leq C\nu_{2}(i_{k})  \|z\|_{Z_{k}} ,  
			\end{align*}
whence $S \in C\nu_{2}(i_{k})B(Z_{k}, E_{K})$. Similarly, we find that for all $z \in Z_n$
			\begin{align*}
				\|R(z)\|_{E_{N}}
				&\leq \sum_{j \in I} |\ev{z}{z_{j}}_{n}| \|g_{j}\|_{E_{N}} \\
				&\leq \varepsilon \left( \sum_{j \in I} |\ev{z}{z_{j}}_{n}|^{2} \right)^{1/2} \left( \sum_{j \in I} \|z_{j}\|^{2}_{Z_{n + 1}} \right)^{1/2} \\
				&= \varepsilon \nu_{2}(i_{n}) \|z\|_{Z_{n}},
			\end{align*}
		and thus $R \in \varepsilon \nu_{2}(i_{n}) B(Z_{n}, E_{N})$.  From \eqref{eq12} we obtain that for all $z \in Z_n$
					\begin{align*} 
				T(z)
				&= \sum_{j \in I} \ev{z}{z_{j}}_{n} e_{j}  = \sum_{j \in I} \ev{z}{z_{j}}_{n} f_{j}  + \sum_{j \in I} \ev{z}{z_{j}}_{n} g_{j}  \\
				&= \sum_{j \in I} \frac{1}{\| z_j\|_{Z_k}}\ev{z}{w_{j}}_{k} f_{j}  + \sum_{j \in I} \ev{z}{z_{j}}_{n} g_{j}  = S(z) + R(z).
			\end{align*}
	\end{proof}

We are ready to show Theorem \ref{t:MainExtensionThm}.

\begin{proof}[Proof of Theorem \ref{t:MainExtensionThm}]
 The sufficiency of  $(\condS)$ follows from Theorem \ref{t:MittagLeffler}, Corollary \ref{c:SpecificCasesAreLocallyEAcyclic}, and Lemma \ref{l:Siss}. Now suppose that $Z$ is weakly $E$-acyclic. By Theorem \ref{t:MittagLeffler}, Corollary \ref{c:SpecificCasesAreLocallyEAcyclic}, and Lemma \ref{l:Proj1=0ImpliesS},  $(E,Z)$ satisfies $(\condWS)$.  In case (a), (b), or (d) holds, Remark \ref{remWSS}(i) implies that  $(E, Z)$ satisfies $(\condS)$. Assume that (c) holds. Corollary \ref{c:wa} yields that $Z \cong k^1(A)$ is weakly acyclic. By \cite[Theorem 5.6]{V-RegPropLFSp}, $Z$ is acylic. Therefore, $(E,Z)$ satisfies $(\condS)$ because of Proposition \ref{l:S=WS+acyclic}.
 \end{proof}
\section{The conditions $(\condA)$ and $(\DOmega)$}\label{sect:sepcond}
Let $E$ be a lcHs with a  fundamental sequence of bounded sets and let $Z$ be an $(LB)$-space. In this section, we discuss separate conditions on $E$ and $Z$  ensuring that the pair $(E, Z)$ satisfies $(\condS)$. We start with the condition $(\condA)$, already mentioned in the introduction.
\begin{definition}
		A lcHs $E$ with a fundamental increasing sequence  of bounded sets $(B_{N})_{N \in \N}$ is said to satisfy $(\condA)$ \cite{V-VektorDistrRandHolomorpherFunk} if
			\[
			\begin{gathered}
				\exists N \in \N ~ \forall M \geq N~\forall  \nu > 0 ~ \exists K \geq M, C > 0 ~ \forall r > 0 \, : \,  \\
				B_{M} \subseteq  r B_{K} + \frac{C}{r^{\nu}} B_{N}.
				\end{gathered}
			\]
		If  $E$ satisfies the previous condition with ``$\forall \nu > 0$" replaced by ``$\exists \nu > 0$", then $E$ is said to satisfy $(\conduA)$ \cite{B-D-V-InterpolVVRealAnalFunc}.
	\end{definition}
	\noindent The conditions $(\condA)$ and  $(\conduA)$ do not depend on the choice of the fundamental increasing sequence of bounded sets $(B_{N})_{N \in \N}$ in $E$.
	
	\begin{remark} \label{remark-DN}
	 Recall \cite{M-V-IntroFuncAnal} that a Fr\'{e}chet space $F$ with a fundamental increasing sequence of seminorms $(\| \,\cdot \,\|_{n})_{n \in \N}$ is said to satisfy $(\DN)$ if 
							\[
	\begin{gathered}
						\exists n \in \N ~ \forall m \geq n ~\forall  \theta \in (0, 1) ~ \exists k \geq m, C > 0   \, : \\  \|x\|_{m} \leq C \|x\|^{\theta}_{n} \|x\|^{1 - \theta}_{k}, \qquad \forall x \in F.
					\end{gathered}
					\]
				If  $F$ satisfies the previous condition with ``$\forall \theta \in (0, 1)$"  replaced by ``$\exists \theta \in (0, 1)$", then $F$ is said to satisfy $(\uDN)$ \cite{M-V-IntroFuncAnal}. The conditions $(\DN)$ and  $(\uDN)$ do not depend on the choice of the fundamental increasing sequence of seminorms $(\| \,\cdot \,\|_{n})_{n \in \N}$ in $F$.
 By \cite[Lemma 1.4]{V-CharUnterraums} and \cite[Proposition 2.2]{B-D-V-InterpolVVRealAnalFunc}, $F$ satisfies $(\DN)$ ($(\uDN)$) if and only if $F^{\prime}$ satisfies $(\condA)$ ($(\conduA)$).
	\end{remark}

The next condition is heavily inspired by the condition $(\Omega)$ for Fr\'echet spaces \cite{M-V-IntroFuncAnal} (see also Section \ref{sec:ExtSESFrechet} below).
\begin{definition}
		An  inductive spectrum of Banach spaces $\mathscr{Z} = (Z_n)_{n \in \N}$ is said to satisfy $(\DOmega)$ if
			\[ 
				\begin{gathered}
					\forall n \in \N ~ \exists m \geq n ~ \forall k \geq m ~ \exists \theta \in (0, 1), C > 0 \, : \\
					 \|z\|_{Z_{m}} \leq C \|z\|_{Z_{n}}^{\theta} \|z\|_{Z_{k}}^{1 - \theta}, \qquad \forall z \in Z_n. 
				\end{gathered}		
			\]
	\end{definition}
\noindent If  $\mathscr{Z}$ and  $\widetilde{\mathscr{Z}}$ are two equivalent inductive spectra of Banach spaces, then  $\mathscr{Z}$ satisfies $(\DOmega)$  if and only if $\widetilde{\mathscr{Z}}$ does so.
Let $Z = \Ind \mathscr{Z}$ be an $(LB)$-space, with  $\mathscr{Z}$ an inductive spectrum of Banach spaces. We say that $Z$ satisfies $(\DOmega)$   if $\mathscr{Z}$ does so. Since all inductive spectra of Banach spaces defining $Z$ are equivalent, this definition does not depend on the choice of the spectrum $ \mathscr{Z}$ of Banach spaces defining $Z$.

\begin{lemma}
		\label{t:A+Omega=E}
		Let $E$ be a lcHs with a  fundamental increasing sequence of bounded sets $(B_{N})_{N \in \N}$ and let $\mathscr{Z} = (Z_n)_{n \in \N}$  be an inductive spectrum of Banach spaces. If $E$ satisfies $(\condA)$ and $\mathscr{Z}$ satisfies $(\DOmega)$, then $(E, \mathscr{Z})$ satisfies $(\condS)$.
	\end{lemma}
	
	\begin{proof}
 We may assume that each $B_N$ is absolutely convex. It suffices to show that, for $n \leq m \leq k$, $N \leq M \leq K$, $C_{0}, C_{1} > 0$, and $\theta \in (0, 1)$ fixed, it holds that 
			\[ \|z\|_{Z_m} \leq C_{0}  \|z\|^{\theta}_{Z_n} \|z\|^{1 - \theta}_{Z_k}, \qquad \forall z \in Z_n,\]
and
	\[ B_{M} \subseteq  r B_{K} + \frac{C_{1}}{r^{\frac{1 - \theta}{\theta}}} B_{N}, \qquad \forall r > 0, \]
	imply that for all $\varepsilon >0$ there is $C >0$ such that 
\[ \|z\|_{Z_m} B_{M} \subseteq C \|z\|_{Z_k} B_{K} +  \varepsilon \|z\|_{Z_n} B_{n}, \qquad  \forall z \in Z_n.\]
Let $\varepsilon > 0$ be arbitrary. For $z \in Z_n \backslash \{0\}$ we set
			\[ r_{z} = \left( \frac{C_1}{\varepsilon} \frac{\|z\|_{Z_m}}{\|z\|_{{Z_n}}} \right)^{\frac{\theta}{1 - \theta}}. \]
			Note that
			\[ r_{z} \leq  \frac{(C_{0}C_1^\theta)^{1/(1 - \theta)}}{ \varepsilon^{\theta/(1-\theta)}} \frac{\|z\|_{{Z_k}}}{\|z\|_{{Z_m}}}, \qquad \forall z \in Z_n \backslash \{0\}.  \]
Consequently,		
			\[ B_{M} \subseteq r_{z} B_{K} + C_{1} r_{z}^{- \frac{1 - \theta}{\theta}} B_{N} \subseteq  \frac{(C_{0}C_1^\theta)^{1/(1 - \theta)}}{ \varepsilon^{\theta/(1-\theta)}}  \frac{\|z\|_{{Z_k}}}{\|z\|_{{m}}} B_{K} + \varepsilon \frac{\|z\|_{{Z_n}}}{\|z\|_{{Z_m}}} B_{N}, \qquad \forall z \in Z_n \backslash \{0\}, \]
		which implies the result.
	\end{proof}

We then have:	
\begin{theorem}
		\label{cor:AO}
		Let $E$ be a locally complete lcHs with a fundamental sequence of bounded sets and let $Z$ be an $(LB)$-space.  Suppose that one of the assumptions (a)-(d) from  Corollary  \ref{c:SpecificCasesAreLocallyEAcyclic}  is satisfied.
		If $E$ satisfies $(\condA)$ and $Z$ satisfies $(\DOmega)$, then $Z$ is weakly $E$-acyclic. 
			\end{theorem}	
	\begin{proof}
	By Lemma \ref{t:A+Omega=E}, $(E,Z)$ satisfies $(\condS)$. In particular, $Z$ is acyclic (Proposition \ref{l:S=WS+acyclic}). The result therefore follows from Theorem  \ref{t:MainExtensionThm} and Proposition \ref{t:MainExtensionThm2}.
	\end{proof}



\section{The splitting of  exact sequences of $(LB)$-spaces} \label{sect:splitting} 
In this short section, we discuss the relation between the notion of weak $E$-acyclicity and the splitting of exact sequences of $(LB)$-spaces. 

\begin{proposition}
		\label{p:Splitting=EAcyclic}
		Let $E$ and $Z$ be $(LB)$-spaces. Then, $Z$ is weakly $E$-acyclic if and only $Z$ is weakly acyclic and  every exact sequence
			\begin{equation}
			 	\label{eq:SplittingSES}
			 	\SES{E}{Y}{Z}{j}{} 
			\end{equation}
		of $(LB)$-spaces splits.
	\end{proposition}
	
	\begin{proof}
		Suppose first that $Z$ is weakly $E$-acyclic. $Z$ is weakly acyclic by Corollary \ref{c:wa}.  Since the sequence \eqref{eq:SplittingSES} is exact,  we have that $j_*: L(Y, E) \to L(E, E)$ is surjective. Hence, there is $T \in L(Y, E)$ such that $T \circ j = \operatorname{id}_E$. This means that $j$ has a left inverse and thus that the sequence \eqref{eq:SplittingSES} splits.  Conversely, suppose that  $Z$ is weakly acyclic and that each exact sequence \eqref{eq:SplittingSES} of $(LB)$-spaces splits. Let 
			\[ \SES{X}{Y}{Z}{\iota}{Q}  \]
			be an arbitrary exact sequence of $(LB)$-spaces and let $T \in L(X, E)$ be arbitrary.  Consider the subspace $L = \{ (\iota(x), T(x)) \mid x \in X \}$ of $Y \times E$. Since $\iota$ is a weak topological embedding (Corollary \ref{p:AcyclicEquiv}) with closed image, $L$ is closed. Hence,  $\widetilde{Y} = (Y \times E) / L$ is an $(LB)$-space. We denote by $q: Y \times E \to \widetilde{Y}$ the canonical quotient map. Define $j : E \rightarrow \widetilde{Y}, e \mapsto q((0, e))$ and let $\pi : \widetilde{Y} \rightarrow Z$ be the continuous linear map such that $\pi(q((y,e))) = Q(y)$ for all $(y,e) \in Y \times E$. We obtain the following exact sequence of $(LB)$-spaces 
			\[ \SES{E}{\widetilde{Y}}{Z}{j}{\pi}.  \]
		Consequently, there is some $P \in L(\widetilde{Y}, E)$ such that $P \circ j = \id_{E}$. Consider the continuous linear map $S: Y \to E, \, y \mapsto P(q((-y,0)))$ and note that $S \circ \iota = T$.
	\end{proof}
	
	\begin{remark}
	Let $E$ and $Z$ be $(LB)$-spaces and assume that $Z$ is weakly acyclic. The fact that every \emph{topologically} exact sequence of the form \eqref{eq:SplittingSES} of $(LB)$-spaces splits does not imply that $Z$ is weakly $E$-acyclic.
As a counterexample, take any $(LB)$-space $Z$ that is weakly acyclic but not acyclic (which exists, see e.g. \cite[\S 5]{V-RegPropLFSp}). Proposition \ref{t:MainExtensionThm2}(i) implies that $Z$ is not weakly $\ell^{\infty}(I)$-acyclic for some index set $I$. Set $E = \ell^{\infty}(I)$. However, by the Hahn-Banach theorem, each topologically short exact sequence of the form \eqref{eq:SplittingSES} automatically splits.
	\end{remark}

	In combination with Theorems  \ref{t:MainExtensionThm} and \ref{cor:AO}, Proposition \ref{p:Splitting=EAcyclic} yields the following two corollaries.

	\begin{corollary}
		\label{t:MainSplittingThm}
		Let $E$ and $Z$  be $(LB)$-spaces and assume that $E$ is locally complete.  Suppose that one of the assumptions (a)-(d) from Theorem  \ref{t:MainExtensionThm} is satisfied.
		Then, the following statements are equivalent:
			\begin{itemize}
			\item[(i)] $Z$ is weakly acylic and every exact sequence \eqref{eq:SplittingSES} of $(LB)$-spaces splits.
			\item[(ii)] $Z$ is weakly $E$-acylic.
				\item[(iii)] $(E, Z)$ satisfies $(\condS)$.
			\end{itemize}
	\end{corollary}
	  	
	\begin{corollary}
		\label{c:Splitting(A)and(DOmega)}
		Let $E$ and $Z$  be $(LB)$-spaces and assume that $E$ is locally complete.  Suppose that one of the assumptions (a)-(d) from  Corollary  \ref{c:SpecificCasesAreLocallyEAcyclic}  is satisfied.
		If $E$ satisfies $(\condA)$ and $Z$ satisfies $(\DOmega)$, then every every exact sequence \eqref{eq:SplittingSES} of $(LB)$-spaces splits.
			\end{corollary}

	%


\section{Surjectivity of tensorized mappings}
\label{sec:ExtSESFrechet}
In this section, we use the results from Sections \ref{sect:main} and \ref{sect:sepcond} to study the surjecitivity of tensorized maps.

Let $F$ be a Fr\'{e}chet space with a fundamental decreasing sequence $(U_{n})_{n \in \N}$ of absolutely convex neighbourhoods of $0$. Then,  $((F^{\prime})_{U^{\circ}_{n}})_{n\in\N}$ is an inductive spectrum of Banach spaces with $F' =  \bigcup_{n \in \N}(F^{\prime})_{U^{\circ}_{n}}$. Recall that $F$ is  distinguished if  $F' =  \Ind ((F^{\prime})_{U^{\circ}_{n}})_{n\in\N}$.  Every quasinormable Fr\'echet space $F$ is distinguished \cite[Corollary 26.19]{M-V-IntroFuncAnal} and, in this case, $F'$ is acylic \cite[Proposition 7.5]{P-HomMethTheoryLCS} (this also follows from \cite[Lemma 26.15]{M-V-IntroFuncAnal}, where it is shown that the inductive spectrum of Banach spaces $((F^{\prime})_{U^{\circ}_{n}})_{n\in\N}$ satisfies condition $(\condQ)$). In the next result, we will use that every Fr\'echet-Schwartz space is Montel and quasinormable.


	
%

	\begin{theorem}
		\label{t:SESExtensionFrechet}
		Let $E$ be a locally complete lcHs with a fundamental sequence of bounded sets and let 
			\begin{equation}
				\label{eq:SESFrechet} 
				\SES{F}{G}{H}{}{Q} 
			\end{equation}
		be an exact sequence of  Fr\'echet-Schwartz spaces. Suppose that one of the following assumptions is satisfied:
			\begin{itemize} 
			\item[(a)] $E \cong k^{\infty}(A)$.
			\item[(b)]   $E$  is an $(LN)$-space.
			\item[(c)]  $F \cong \lambda_0(B)$.
			\item[(d)]  $F$ is nuclear.
			\end{itemize}
		If $(E, F')$ satisfies $(\condWS)$, then the map $Q\varepsilon \operatorname{id}_E: G \varepsilon E \to H\varepsilon E$ is surjective.
	\end{theorem}
	
	\begin{proof} 
		Since $G$ and $H$ are Montel, we have that $G \varepsilon E = L(G', E)$ and $H \varepsilon E = L(H', E)$. Hence, $Q\varepsilon \operatorname{id}_E = (Q^t)_*: L(G', E) \to L(H',E)$.
		As \eqref{eq:SESFrechet} is an exact sequence of distinguished  Fr\'echet spaces, its dual sequence		
		\[ \SES{H'}{G'}{F'}{Q^{t}}{}\]
		consists of $(LB)$-spaces and is exact. 
		Therefore, it suffices to prove that $F'$ is weakly $E$-acyclic. Note that  $(E, F')$ satisfies $(\condS)$ as $F'$ is acyclic (Proposition \ref{l:S=WS+acyclic}). The result therefore follows from Theorem  \ref{t:MainExtensionThm} and Proposition \ref{t:MainExtensionThm2}.	\end{proof}

\begin{definition}
Let $F$ be a Fr\'echet space with a fundamental decreasing sequence $(U_{n})_{n \in \N}$ of absolutely convex neighbourhoods of $0$. Then, $F$ is said to satisfy $(\Omega)$ \cite{M-V-IntroFuncAnal}  if the inductive spectrum of Banach spaces  $((F^{\prime})_{U^{\circ}_{n}})_{n\in\N}$  satisfies $(\DOmega)$. 
\end{definition}
\noindent The condition $(\Omega)$ does not depend on the choice of the fundamental decreasing sequence $(U_{n})_{n \in \N}$ of absolutely convex neighbourhoods of $0$ in $F$.
Lemma \ref{t:A+Omega=E} and Theorem \ref{t:SESExtensionFrechet}  yield the following result.	
		
		\begin{theorem}\label{t:OmegaA}
		Let $E$ be a locally complete lcHs with a fundamental sequence of bounded sets and let \eqref{eq:SESFrechet}  be an exact sequence of Fr\'echet-Schwartz spaces. Suppose that one of the assumptions (a)-(d) from Theorem \ref{t:SESExtensionFrechet} is satisfied. If $E$ satisfies ($\condA$) and $F$ satisfies  $(\Omega)$, then the map $Q\varepsilon \operatorname{id}_E: G \varepsilon E \to H\varepsilon E$ is surjective.
	\end{theorem}

	\section{Vector-valued Eidelheit families}\label{sect:Eidelheit}
Fix a countably infinite index set $J$. Let $G$ be a Fr\'echet space. Eidelheit  \cite{E-TheorieSystemeLinGleichungen} (see also \cite[Theorem 26.27]{M-V-IntroFuncAnal}) characterized the families  $(x'_{j})_{j \in J} \subseteq G'$ such that the infinite system of scalar-valued linear equations
$$
\langle x'_j,x \rangle  =a_j, \qquad j \in J,
$$
has a solution $x \in G$ for each $(a_j)_{j \in J} \in \C^J$. In such a case, the family $(x'_{j})_{j \in J}$ is called \emph{Eidelheit}. By definition, $(x'_{j})_{j \in J}$ is Eidelheit if and only if the map
	\[ Q = Q_{(x'_{j})_{j \in J}}: G \rightarrow \C^{J} , \,  x \mapsto (\ev{x'_{j}}{x})_{j \in J} \]
is surjective. Let $E$ be a lcHs. We define the \emph{associated $E$-valued family of  $(x'_{j})_{j \in J}$} as 
$$
(x'_{j} \varepsilon \id_E)_{j \in J} \subseteq L( G \varepsilon E, E),
$$
where we identified $\C \varepsilon E$ with $E$.
This family is called \emph{Eidelheit} if  the infinite system of $E$-valued  linear equations
$$
\langle x'_j \varepsilon \id_E,x \rangle  =e_j, \qquad j \in J,
$$
has a solution $x \in G\varepsilon E$ for each $(e_j)_{j \in J} \in E^J$. Note that $(x'_{j} \varepsilon \id_E)_{j \in J}$  is Eidelheit if and only if the map
	\[ Q \varepsilon \id_E: G \varepsilon E \rightarrow \C^{J} \varepsilon E \cong E^J \]
is surjective. 

We have the following natural problem: 
\emph{Let $(x'_{j})_{j \in J} \subseteq G^{\prime}$  be Eidleheit and let $E$ be a lcHs. Find
sufficient and necessary conditions such that the associated $E$-valued family of  $(x'_{j})_{j \in J}$ is also Eidelheit.}
	
	If $G$ is a nuclear Fr\'echet space, then for each Fr\'echet space $E$ the associated $E$-valued family of any Eidelheit family $(x'_{j})_{j \in J} \subseteq G^{\prime}$ is again Eidelheit: This follows
from the equality $G \varepsilon E = G \widehat{\otimes}_\pi E$ \cite[Proposition 1.4]{K-Ultradistributions3} and the fact that the $\pi$-tensor product of two surjective
continuous linear maps between Fr\'echet spaces is again surjective \cite[Proposition 43.9]{T-TopVecSpDistKern}. 

The goal of this section is to study the above problem if $E$ is a lcHs with a fundamental sequence of bounded sets. Since a closed subspace of a nuclear space is again nuclear, Theorem  \ref{t:OmegaA} directly yields the following result.
	\begin{proposition}\label{suff-Eidelheit}
        Let $G$ be a nuclear Fr\'echet space. Let $(x'_j)_{j \in J} \subseteq G'$ be an Eidelheit family such that 
	$$
	\ker Q = \{ x \in G \, | \, \langle x'_j, x \rangle = 0, \, \forall j \in J \} 
	$$
	satisfies $(\Omega)$. Let $E$ be a locally complete lcHs with a   fundamental sequence of bounded sets that satisfies $(\condA)$. Then, the associated $E$-valued  family of  $(x'_{j})_{j \in J}$ is also Eidelheit.
	\end{proposition}
	\begin{remark}
	In general $\ker Q$ does not satisfy $(\Omega)$ even if $G$ does so \cite{V-Mit}. We refer to  \cite{B-LinTopStructClosedIdealsFAlg,V-Kernel,V-K2} for various situations in which $\ker Q$ does inherit $(\Omega)$ from $G$.
	\end{remark}
	%
	Next, we focus on the necessity of the condition  $(\condA)$. The following result was shown by Vogt in the special case $G \cong \lambda^{1}(A)$ \cite[Lemma 4.4]{V-TensorFundDFRaumFortsetz}. We present here a more general version with an alternative proof. 
	
			\begin{proposition}
	\label{t:NecessityA}
			 Let $G$ be a Fr\'echet space satisfying $(\uDN)$ and let $(x'_j)_{j \in J} \subseteq G'$ be Eidelheit. Let $E$ be a lcHs with a fundamental sequence of bounded sets. If the associated $E$-valued  family of  $(x'_{j})_{j \in J}$ is Eidelheit, then
 $E$ satisfies $(\condA)$.
			 \end{proposition}
		
		\noindent In view of Remark \ref{remark-DN} and the fact that $G \varepsilon E = L(G'_c, E) \subseteq L(G',E)$, Proposition \ref{t:NecessityA} is a consequence of the following result.

		\begin{proposition}
		\label{p:NecessityA}
		Let $E$ and $Y$ be two lcHs with a fundamental sequence of bounded sets and assume that $Y$ satisfies ($\conduA$). Suppose that there exists a sequence $(y_j)_{j \in J} \subseteq Y$ such that for each sequence $(e_j)_{j \in J} \in E^J$ there is $T \in L(Y,E)$  such that $T(y_j) = e_j$ for all $j \in J$. Then, $E$ satisfies $(\condA)$.
\end{proposition}

\begin{proof} We may suppose that $J = \N$.
Let $(D_{n})_{n \in \N}$ be a fundamental increasing sequence of bounded sets in $Y$. Since $Y$ satisfies $(\conduA)$, there is $n_0 \in \N$ such that for all $n \geq n_0$
\[
			\begin{gathered}
				\exists k(n) \geq n , \nu(n), C(n) > 0 ~ \forall r > 0 \, \exists v_{n,r} \in r D_{k(n)},  w_{n,r} \in \frac{C(n)}{r^{\nu(n)}} D_{n_0} \, : \, y_n = v_{n,r} + w_{n,r}.
				\end{gathered}
\]
Let $(B_{N})_{N \in \N}$ be a fundamental increasing sequence of bounded sets in $E$ such that each $B_N$ is a Banach disk. The fact that $E$ satisfies $(\condA)$ is equivalent to (cf.\ \cite[Lemma 3.1 and Bemerkung 3.1']{V-VektorDistrRandHolomorpherFunk}) 
			$$
				\exists N \in \N, \nu > 0 ~ \forall e \in E ~ \exists K \geq N, C > 0 ~ \forall r > 0 \, : \,  e \in  r B_K + \frac{C}{r^{\nu}} B_{N}. 
			$$ 		
Suppose now that $E$ does not satisfy $(\condA)$.  Then, by the above, there is a sequence $(e_n)_{n \geq n_0} \subseteq E$ such that 
	\begin{equation} 
				\label{eq:AlternativeAcor}
\forall n \geq n_0 ~ \forall K \geq n, C > 0 ~ \exists r > 0 \, : \,  e_n \notin  r B_K + \frac{C}{r^{\nu(n)}} B_{n}. 
			\end{equation} 
Let $T \in L(Y,E)$ be such that $T(y_n) = e_n$ for all $n \geq n_0$. Pick $N_0 \geq n_0$ and $K \geq N_0$ such that $T(D_{n_0}) \subseteq B_{N_0}$ and $T(D_{k(N_0)}) \subseteq B_K$. Then, for all $r >0$
$$
e_{N_0} = T(y_{N_0}) = T(v_{N_0,r}) + T(w_{N_0,r}) \in rB_K + \frac{C({N_0})}{r^{\nu({N_0})}} B_{N_0},
$$
which contradicts \eqref{eq:AlternativeAcor}.
\end{proof}

\begin{remark}
Set $\C^{(J)} = \bigoplus_{J} \C$. By using the canonical isomorphism $L(\C^{(J)}, E) \cong E^{J}$, $E$ a lcHs, Proposition \ref{t:NecessityA}  may be reformulated as follows: \emph{Let $E$ and $Y$ be two lcHs with a fundamental sequence of bounded sets and assume that $Y$ satisfies  ($\conduA$).  Let $\iota : \C^{(J)} \to Y$ be a continuous  linear map. If $\iota_*: L(Y, E) \to L(\C^{(J)} ,E)$ is surjective, then $E$ satisfies $(\condA)$.} Hence,  Proposition \ref{t:NecessityA}  may be interpreted as a partial  converse to Theorem \ref{cor:AO}.
\end{remark}

By combining Propositions \ref{suff-Eidelheit} and \ref{t:NecessityA}, we obtain the following result.
 \begin{theorem}
	\label{t:Eid}
	Let $G$ be a nuclear Fr\'echet space satisfying $(\uDN)$. Let $(x'_j)_{j \in J} \subseteq G'$ be an Eidelheit family such that 
	$$
	\ker Q = \{ x \in G \, | \, \langle x'_j, x \rangle = 0, \, \forall j \in J \} 
	$$
	satisfies $(\Omega)$. Let $E$ be a locally complete lcHs with a   fundamental sequence of bounded sets. Then, the associated $E$-valued  family of  $(x'_{j})_{j \in J}$ is Eidelheit if and only if $E$ satisfies $(A)$.
			 \end{theorem}
			 
	Let $E$ be a locally complete lcHs with a   fundamental sequence of bounded sets. Theorem \ref{t:Eid} may be used to characterize the validity of the $E$-valued version of many classical problems in analysis by the condition $(\condA)$ on $E$. We present two examples: E. Borel's theorem and interpolation by holomorphic functions. We denote by $C^\infty(\R^d;E)$ the space of $E$-valued smooth functions on $\R^d$. We set   $C^\infty(\R^d) = C^\infty(\R^d;\C)$.
			 
			 \begin{example}
			 \label{ex:vvBorel}
			\emph{ Let $E$ be a locally complete lcHs with a fundamental sequence of bounded sets. Then,  the following statements are equivalent:
			\begin{itemize}
			\item[(i)] For each sequence $(e_\alpha)_{\alpha \in \N^d} \subseteq E$ there is $f \in C^\infty(\R^d;E)$  such that $f^{(\alpha)}(0) = e_\alpha$ for all $\alpha \in \N^d$.
			\item[(ii)]  $E$ satisfies $(A)$.
			\end{itemize}}
			 \end{example}
			 \begin{proof}
			 Set $L = [-1,1]^d$ and let $\mathcal{D}_L(E)$ be the space consisting of all $f \in  C^\infty(\R^d;E)$ with $\operatorname{supp} f \subseteq L$.  We set $\mathcal{D}_L = \mathcal{D}_L(\C)$. By multiplying with a suitable cut-off function, one sees that it suffices to show that $E$ satisfies $(\condA)$ if and only if for each sequence $(e_\alpha)_{\alpha \in \N^d} \subseteq E$ there is $f \in \mathcal{D}_L(E)$  such that $f^{(\alpha)}(0) = e_\alpha$ for all $\alpha \in \N^d$. We endow  $C^\infty(\R^d)$ and  $\mathcal{D}_L$ with their natural Fr\'echet space topology.  For $\alpha \in \N^d$ we define $x'_\alpha = (-1)^{|\alpha|}\delta^{(\alpha)} \in  \mathcal{D}'_L$, i.e., $\langle x'_\alpha, f \rangle = f^{(\alpha)}(0)$, $f \in \mathcal{D}_L$. By E. Borel's theorem (see e.g.\ \cite[Theorem 26.29]{M-V-IntroFuncAnal} for the case $d =1$),  $(x'_{\alpha})_{\alpha \in \N^d} \subseteq  \mathcal{D}'_L$ is Eidelheit. Since $E$ is locally complete,  the spaces   $C^\infty(\R^d;E)$ and   $C^\infty(\R^d) \varepsilon E$ are canonically isomorphic via the map  (cf.\ \cite{BFJ})
			 $$
			 \Phi: C^\infty(\R^d;E) \to C^\infty(\R^d) \varepsilon E \cong L(E'_c,C^\infty(\R^d)), \, f \mapsto (e' \mapsto \langle e', f( \,\cdot\,) \rangle).
			 $$
			 Consequently, 	$\Phi: \mathcal{D}_L(E) \to \mathcal{D}_L \varepsilon E$ is an isomorphism as well. Note that, under this isomorphism, 	$x'_\alpha \varepsilon \operatorname{id}_E(f) = f^{(\alpha)}(0)$, $f \in \mathcal{D}_L(E)$. Hence, it suffices to show that  the associated $E$-valued  sequence of  $(x'_{\alpha})_{\alpha \in \N^d}$ is Eidelheit if and only if $E$ satisfies $(A)$.  We have that $\mathcal{D}_L$ is isomorphic to the Fr\'echet space of rapidly decreasing sequences $s$ (see e.g. \cite[p.\ 363]{M-V-IntroFuncAnal} for the case $d =1$). Hence, $\mathcal{D}_L$  is nuclear and satisfies $(\DN)$ (and thus $(\uDN)$). Furthermore, 
			 $$
	\ker Q = \{ f\in \mathcal{D}_L \, | \,   f^{(\alpha)}(0)= 0, \, \forall \alpha \in \N^d \} 
	$$
	satisfies $(\Omega)$ by \cite[Satz 2.2]{Tidten}. The result now follows from Theorem \ref{t:Eid}.			 
			 \end{proof}
			  Given $U \subseteq \C$ open and a locally complete lcHs $E$, we denote by $\mathcal{O}(U;E)$ the space of $E$-valued holomorphic functions on $U$.
			  We set   $\mathcal{O}(U) = \mathcal{O}(U;\C)$.

			  \begin{example}
			  \label{ex:vvInterpol}
			\emph{ Let $E$ be a locally complete lcHs with a fundamental sequence of bounded sets. Let $U \subseteq \C$ be open and connected. Let $(u_n)_{n \in \N} \subseteq U$ be a sequence without accumulation points in $U$ and let $(k_n)_{n \in \N}$ be a sequence of natural numbers.
			\begin{itemize}
			\item[(i)] For each family $(e_{n,k})_{n \in \N,  k \leq k_n} \subseteq E$ there is $f \in \mathcal{O}(U;E)$  such that $f^{(k)}(u_n) = e_{n,k}$ for all $n \in \N$ and $k = 0, \ldots, k_n$.
			\item[(ii)]  $E$ satisfies $(A)$.
			\end{itemize}}
			 \end{example}
\begin{proof}
We endow $\mathcal{O}(U)$ with its natural Fr\'echet space topology.  For $n \in \N$ and $k \leq k_n$ we define $x'_{n,k} = (-1)^{k}\delta^{(k)}(\, \cdot \, - u_n) \in  \mathcal{O}'(U)$, i.e., $\langle x'_{n,k}, f \rangle = f^{(k)}(u_n)$, $f \in \mathcal{O}(U)$. By 
a standard interpolation result for holomorphic functions \cite[Theorem 15.13]{Rudin},  $(x'_{n,k})_{n \in \N,  k \leq k_n}\subseteq  \mathcal{O}'(U)$ is Eidelheit.
Since $E$ is locally complete,  the spaces   $ \mathcal{O}(U;E)$ and   $ \mathcal{O}(U)\varepsilon E$ are canonically isomorphic via the map (cf.\ \cite{BFJ})
			 $$
			\mathcal{O}(U;E) \to \mathcal{O}(U)\varepsilon E  \cong L(E'_c,\mathcal{O}(U)), \, f \mapsto (e' \mapsto \langle e', f( \,\cdot\,) \rangle).
			 $$
			 Note that, under this isomorphism, 	$x'_{n,k} \varepsilon \operatorname{id}_E(f) = f^{(k)}(u_n)$, $f \in \mathcal{O}(U;E)$. Hence, it suffices to show that the associated $E$-valued  sequence of    $(x'_{n,k})_{n \in \N,  k \leq k_n}$ is Eidelheit if and only if $E$ satisfies $(A)$.  We have that $\mathcal{O}(U)$ is a nuclear Fr\'echet space that satisfies  $(\uDN)$ \cite[Satz 5.1]{V-fin} ($U$ is connected). Furthermore, 
			 $$
	\ker Q = \{ f\in \mathcal{O}(U) \, | \,   f^{(k)}(u_n)= 0, \, \forall n \in \N, k = 0, \ldots, k_n \} 
	$$
	is isomorphic to $\mathcal{O}(U)$ via the map
	$$
	\mathcal{O}(U) \to \ker Q, \, f \to gf
	$$ 
	where $g \in \mathcal{O}(U)$ is such that $g$ has a zero of order $k_n$ at $u _n$ for each $n \in \N$ and $g$ has no other zeroes \cite[Theorems 15.9 and 15.11]{Rudin}.
	 Consequently,  	$\ker Q$ satisfies $(\Omega)$ as  $\mathcal{O}(U)$ does so \cite[Corollary 4.3(a)]{Petzsche}. The result now follows from Theorem \ref{t:Eid}.			 

\end{proof}


	\begin{remark}
	We refer to \cite{B-D-V-InterpolVVRealAnalFunc} for results concerning $E$-valued interpolation by real analytic functions on open subsets of $\R$, where $E$ is a sequentially complete $(DF)$-space $E$. Since the space of real analytic functions on an open subset of $\R$ is a $(PLS)$-space (and not a Fr\'echet space), this problem is much more difficult and deeper than Example   \ref{ex:vvInterpol} (in this regard, see also \cite{D-V-InfSystLinEqRealAnFunc}).	
	\end{remark}

\end{document}